\documentclass[a4paper, 11pt]{amsart}
\usepackage[utf8]{inputenc}

\usepackage{canonicalpaper}

\usepackage{graphicx}
\usepackage{todonotes}

\numberwithin{equation}{section}
\numberwithin{thmglobal}{section}

\noarxiv{
    \addbibresource{bibLaTeX.bib}
}

\arxiv{
\let\textcite\cite
}

\title{Curvature bounds for regularized Riemannian metrics}

\author[D.~Luckhardt]{Daniel Luckhardt$^{*}$}
\address[D.~LUCKHARDT]{Department of Mathematics, Ben-Gurion University of the Negev, Israel.}
\email{luckhard@post.bgu.ac.il}

\thanks{${}^{*}$Supported by the ISF project ``Action now: geometry and dynamics of group actions''}

\author[J.-B.~Korda\ss]{Jan-Bernhard Korda\ss${}^{\S}$}
\address[J.-B.~KORDA\ss]{Département de mathématiques, Université de Fribourg, Switzerland.}
\email{jan-bernhard.kordass@unifr.ch}

\thanks{${}^{\S}$Supported by the SNSF-Project 200021E-172469 and the DFG-Priority programme Geometry at infinity (SPP 2026)}
\date{October 2019}

\begin{document}

\maketitle

\begin{abstract}
    We investigate regularization of riemannian metrics by mollification.
    Assuming both-sided bounds on the Ricci tensor and a lower injectivity radius bound we obtain a uniform estimate on the change of the sectional curvature. 
    Actually, our result holds for any metric with a uniform bound on the $\Sobolev{2}[p]$-harmonic radius.
\end{abstract}

\section{Introduction}

The goal of this note is to show that regularization by a naive mollification of a riemannian metric satisfying certain geometric conditions can be set up to control the deviation of the sectional curvature.

A riemannian metric on a smooth manifold $M$ is a section in the symmetric $(0,2)$-tensor bundle over $M$.
The regularity of this section is then referred to as the regularity of the metric.
Under certain geometric limit processes, it is a common phenomenon to loose a controlled level of regularity.
Gromov-Hausdorff limits of isometry classes of smooth riemannian metrics that satisfy certain curvature bounds, are a prominent example thereof.
We will briefly recall fundamental results in this area in \cref{subsec:precptness-thms}.
In this work we are interested in procedures to regain regularity and at the same time control the deviation of curvature.
Various such techniques have been studied with a view towards different goals (e.g.\ \cite{CheegerGromov85, Abresch88}).

A fundamental tool in deriving and phrasing our results are \emph{chart norms}, that control the regularity properties of the metric tensor and its derivatives in a given chart.
Let $\psi \colon (\Openball{0}{r}, 0) \to (M,\Basepoint)$ be a harmonic chart in a smooth, pointed, $n$-dimensional riemannian manifold $(M^n,g,\Basepoint)$, where we denote by $\Openball{0}{r} \subset \R^n$ the open ball of radius $r$ with respect to the euclidean norm.
Recall that a chart is called \emph{harmonic}, if the coordinate-functions of $\psi$ are harmonic with respect to the Laplace-Beltrami operator of $g$.
The harmonic chart norm $\|\psi\|^{\textnormal{harm}}_{W^{m,p},r}$ is bounded by $Q \geq 0$, if $Q$ gives control of the derivatives of the metric tensor, its inverse and its first $m$ derivatives in the $\Lebesgue{p}$-norm.
We refer to \cref{subsec:chart-norms} for a more detailed explanation.

Unless otherwise stated, we will always denote by $M$ a smooth manifold of dimension $n$.
Given a collection of charts $\{\psi_i\colon \Openball{0}{r} \to M\}_{i\in I}$, a partition of unity $\{\rho_i\}_{i\in I}$, and a fixed mollifier function $\varphi_t$ for $t \in (0,T]$, we define a mollified riemannian metric $g^{[t]}$ (cf.\ \cref{def:mollify_g}).
This metric is defined on $ \bigcup_{i\in I} \psi_i(\Openball{0}{\nicefrac{r}{2}}) $ and, provided bounds on the harmonic chart norms, has a curvature tensor that can be controlled as follows:

\begin{restatable}{msatz}{ThmMollifiedCurv}
        \label{thm:mollifiedCurv convexCombi}
    Given $p > 2n $ and $r, Q> 0$.
    Choose $\beta \in (0, 1 - \nicefrac{2n}{p})$.
    Then there is some $\flowVarMax \in (0, \nicefrac{r}{2}) $ such that for any pointed  smooth riemannian manifold $(M, g, \Basepoint)$ and any finite collection of charts with a corresponding partition of unity
    \[
        \{\psi_i\colon (\Openball{0}{r}, 0) \to (M,\psi_i(0))\}_{i\in I}
    ,\quad
        \vec{\varrho} \coloneqq \{\varrho_i\colon M \to [0,1]\}_{i\in I}
    \]
    we have that
    for any section $\vec{v} \in (\Tangent \psi( \Openball{0}{r}))^{\times 3} \times \Tangent^* \psi( \Openball{0}{r}) $ 
    \begin{equation*}\label{eq:mollifiedCurv convexCombi}
            \left(\forall i\in I\colon \|\psi_i\|_{\Sobolev{2}[p], r}^{\textnormal{harm}} \leq Q\right)
        \implies
                \Riem_{ g^{ [\flowVar ] } }(x) (\vec{v})
              - 
                \sup_{\Openball{x}{e^Q \flowVar }}
                    \Riem_g (\vec{v})
            \leq 
                C \|\vec{v}\|_g \flowVar^\beta
    \end{equation*}
    for any $ \flowVar \in (0,\flowVarMax] $, $
            x
        \in 
            \bigcap\limits_{i\in I} 
                \psi_i(\Openball{0}{\nicefrac{r}{2}}) 
    $ and with
    \[ 
            C
        = 
            C\left(n,p,r,Q, \beta, 
                \{\|\varrho_i \circ \psi_i \|_{\Hoeld{2}}\}_{i\in I},
                \# I
            \right)
    .
    \]
\end{restatable}
In the theorem, we denote by $\|\vec{v}\|_g$ the semi-norm obtained as the supremum of the point-wise product s $\|v_1(x)\|_{g(x)} \cdots \|v_4(x)\|_{g(x)}$ (cf.\ \cref{sec:mollified-curvature-tensor} for further details).

The new metric $g^{[t]}$ is obtained by applying a mollification operator $P_t$ in the charts $\psi_i$.
We can summarize the ansatz of our proof of this statement as follows:
\begin{align}
\label{eq:ansatz}
        \Riem_{g^{ [\flowVar ] }} (x)  (\vec{v})
    &\leq
            P_\flowVar (\Riem_g  (\vec{v}))(x)
        + 
            | 
                \Riem_{g^{ [\flowVar ] }} (x) (\vec{v}) 
              - P_\flowVar (\Riem_g  (\vec{v})) (x) 
             |
\\
\nonumber
    &\leq 
            \sup_{y\in\psi(\Openball{0}{\flowVarMax})} \Riem_g (y) (\vec{v})
        + 
            \| 
                \Riem_{g^{ [\flowVar ] }} (\vec{v}) 
              - P_\flowVar \Riem_g  (\vec{v})
            \|_{\Lebesgue{\infty}}
\end{align}
where in the last step we used that convolution does not increase the supremum of a function.
Hence it suffices to find an $\Lebesgue{\infty}$-estimate for the ``commutator''
\[
          \Riem_{ g^{ [\flowVar ] } } - P_\flowVar \Riem_g 
    \text{ ``}{=}\text{ }
        [ \Riem, P_\flowVar ]g \text{''}
\]
which will be obtained in \cref{lem:commutatorLikeEstimate}.

Finally, we state our main result on the deviation of the sectional curvature.
Define for a riemannian manifold $(M, g)$
\begin{align*}
        \maxSec(g)
    &\coloneqq
        \sup \Sec_g(x)(v,w)
,
\\
        \minSec(g)
    &\coloneqq
        \sup \Sec_g(x)(v,w)
\end{align*}
where the suprema are taken over all $ x \in M $ and $v,w \in \Tangent_x M$ such that $ \langle v, w \rangle_g > 0 $.
Moreover, we denote by $\|(M,g)\|_{\Sobolev{2}[p], r}^{\textnormal{harm}}$ the lowest bound $Q$ such that for each point $\Basepoint \in M $ there is a harmonic chart $\psi\colon \Openball{0}{r} \to M$ mapping $0$ to $\Basepoint$ such that the metric tensor $\psi^* g$ on $\Openball{0}{r}$ is controlled by $Q$ with respect to the $\Sobolev{2}[p]$-norm.

\begin{restatable}{msatz}{ThmMollifiedSec}\label{thm:mollifiedSec}
    Let $p > 2n $, $r, Q> 0$, and 
    $\beta\in(0,1- \nicefrac{2n}{p}) $.
    Then there is some $ \flowVarMax > 0$ such that for any smooth riemannian manifold $(M,g)$
    with $\|(M,g)\|_{\Sobolev{2}[p], r}^{\textnormal{harm}} \leq Q$
    there is a locally finite cover of charts
    \[
        \psi_i\colon (\Openball{0}{r}, 0) \to (M,\psi_i(0)),
    \]
    indexed by $ i \in I $, such that the mollified metrics $  
            g^{ [\flowVar] }
    $ have at any $x\in M$ sectional curvature $ \Sec_{ g^{ [\flowVar]} }(x) $ in the interval
    \begin{equation}\label{eq:secInterval}
            \left[
                \minSec( g_\dbBlank|_{\Openball{x}{\flowVar}}),
                \maxSec( g_\dbBlank|_{\Openball{x}{\flowVar}})
            \right]
            \cdot
            [1-C \flowVar, 1+C \flowVar]
        +
            [-C\flowVar^\beta, C\flowVar^\beta]
    \end{equation}
    for all $\flowVar \in (0,\flowVarMax]$,
    where $C = C(n,Q,r,p)$.
    Moreover, we have that
    \begin{equation}\label{eq:secChartNorm}
        \|M\|_{\Hoeld{2}[\beta], r} \leq Q'
    \end{equation}
    for some $
            Q' = Q'( n, \beta, Q)
    $.
\end{restatable}

The advantage of the regularized metric is especially that is admits pointwise control on the curvature tensor that is stable under Gromov-Hausdorff limits, i.e.\ as apparent from \cref{eq:fundamentalThm} below, a Gromov-Hausdorff limit of regularized metrics is again a riemannian manifold with $ \|M\|_{\Hoeld{2}[\beta], r} \leq Q' $.

\begin{restatable}{mkor}{corInjRad}\label{cor:injRad}
    Let $ \iota > 0 $ and $\kappa \geq 0$.
    Then there exist $r>0$ and  $ \flowVarMax > 0 $ such that for any
    smooth riemannian manifold $(M, g)$ with 
    \[
        \injRad(M) \geq \iota
    ,\quad
        \|\Ricci M\|_{\Lebesgue{\infty}} \leq \kappa
    \]
    there exist charts $\psi_i\colon \Openball{0}{r} \to M$ such that the mollified metrics $ 
            g^{ [\flowVar ] } 
    $ 
    have at any $x\in M$ sectional curvature $ \Sec_{ P_{\flowVar } g }(x) $ in the interval
    \cref{eq:secInterval}
    for all $\flowVar \in (0,\flowVarMax]$,
    where $C = C(n,\iota,\kappa)$.
    Moreover, the norm $ \|M\|_{\Hoeld{2}[\beta], r} $ is bounded in terms of $n, \beta$ and $ Q$.
\end{restatable}

\begin{restatable}{mkor}{corRegidity}\label{cor:rigidity}
    \begin{subequations}
    Let $\delta \in (0,1)$ and assume one of the conditions
    \cref{eq:topAssumptionAspher,eq:topAssumptionEvenDim,eq:dim3,eq:quaterPinched}:
    \begin{empheq}[left={\Sec_g \in [\delta, 1] \text{ and }\empheqbiglbrace~}]{align}
     \label{eq:topAssumptionAspher} &
        \#\pi_2 (M) < \infty \text{, or}
    \\
    \label{eq:topAssumptionEvenDim} &
        \dim M \text{ is even, or}
    \\
    \label{eq:quaterPinched} &
        \delta \geq \nicefrac{1}{4}  - \varepsilon', \quad \text{where }\varepsilon' \approx 10^{-6} 
    \end{empheq}
    \begin{gather}
         \Sec_g \leq 1,
         \Ricci_g \geq \delta,
         \text{ and }
         \dim M = 3.
    \label{eq:dim3}
    \end{gather}
    Then there exist charts $\psi_i\colon \Openball{0}{r} \to M$ such that the mollified metrics $ 
            g^{ [\flowVar ] } 
    $ 
    have at any $x\in M$ sectional curvature $ \Sec_{ P_{\flowVar } g }(x) $ in the interval
    \cref{eq:secInterval}
    for all $\flowVar \in (0,\flowVarMax]$,
    where $C = C(n,\kappa)$.
    Moreover, the norm $ \|M\|_{\Hoeld{2}[\beta], r} $ is bounded in terms of $n, \beta$ and $ Q$.
    \end{subequations}
 \end{restatable}

As mentioned before, there are other techniques towards a regularization of riemannian metrics.
In \cite{Abresch88}, Abresch constructs a smoothing operator $S_{\varepsilon}$, which satisfies $\|\nabla^m\Riem_{S_{\varepsilon}g}\| \leq C \frac{1}{(1+\varepsilon \Lambda)^{m+2}}$, where $C = C(n)$ is a constant and $\Lambda$ is a bound on the sectional curvature.
In particular, $S_\varepsilon $ preserves isometries of the original metric.
Similarly, it is known that given a bound on the sectional curvature, the application of Ricci flow amounts to a regularization of the metric, which gives $\|\nabla^m\Riem_{g(t)}\| \leq C(n, m, t)$ \cite{BemelmansMin-OoRuh84,Shi89,Rong96,Kapovitch05,He16}.
In contrast to our results, this again preserves the isometries of the original metric, but just as Abresch's result requires stronger bounds on the curvature.
The crucial point is that both methods do not provide a bound on the difference $\|\Riem_g - \Riem_{S_{\varepsilon}g} \|$ or $\|\Riem_g - \Riem_{g(t)} \|$.

We are motivated by the viewpoint of moduli spaces of riemannian metrics with pinched curvature and regard \cref{thm:mollifiedSec} as a result on a controlled perturbation within such a space.

The paper is structured as follows:
In \cref{sec:review-mollification-chart-norms} we will recall mollification, chart norms, and known uniform bounds on the regularity of the metric tensor under geometric conditions.
The subsequent \cref{sec:mollified-curvature-tensor} gives a proof of the main technical tool, \cref{thm:mollifiedCurv convexCombi}, beginning with a local version of its statement.
This is then used in \cref{sec:consequence-for-sec} to derive the main result, \cref{thm:mollifiedSec}, and its corollaries.

\section{Review of mollification and chart norms}\label{sec:review-mollification-chart-norms}

We will give a short introduction to H\"older spaces and mollification.
In the subsequent two subsections we will explain norm bounds for a riemannian metric that are independent of a distinguished coordinate system, and state the fundamental examples and properties of such systems.

\subsection{H\"older spaces}

Besides $\Lebesgue{p}$-classes we will use H\"older spaces of functions $ f\colon \R^n \to \R^N $.
\begin{subequations}
For $m = 0, 1, \ldots$ define
\begin{equation}\label{eq:nabla}
    \nabla^m f\colon \Omega \to \R^{N\cdot n^m} 
\end{equation}
to be the function of all derivatives of order $m$ and 
\begin{equation}\label{eq:nabla_leq}
        \nabla^{\leq m} f
    \colon 
        \Omega \to \R^{N\cdot n^0 + \ldots +  N\cdot n^m}
\end{equation}
to be the collection of all derivatives of order 0 to order $m$.
\end{subequations}
\begin{subequations}
Further let
\begin{equation}\label{eq:maxNorm}
        |(x_1, \ldots, x_n)|
    \coloneqq
        \max\{|x_1|, \ldots, |x_n|\}
\end{equation}
for $ x = (x_1, \ldots, x_n) \in \R^n$ (which is in contrast to the euclidean norm $  \dist{0}{x} $).
Recall that the H\"older norm---for $\alpha \in [0,1]$, $m = 0, 1, \ldots$, and a domain $\Omega\subset\R^n$---is given by
\begin{align}
\label{eq:Holder_norm}
        \|f\|_{\Hoeld{m}[\alpha]}
    =
        \|f\|_{\Hoeld{m}[\alpha](\Omega)}
    &\coloneqq
        \|f\|_{\Hoeld{m}} + 
        \sum_{k=0}^m \|\nabla^k f\|_\alpha 
\intertext{where $\|f\|_\alpha  = 0 $ in case $\alpha = 0 $ and otherwise }
\label{eq:Holder_seminorm}
        \|f\|_\alpha 
    &\coloneqq
        \sup_{ \substack{ x, y \in \Omega \\ x \neq y } } 
            \frac{| f(x) - f(y) |}{|x-y|^\alpha}
.
\end{align}
\end{subequations}
Denote by $
    \Hoeld{m}[\alpha](\Omega) = \Hoeld{m}[\alpha](\Omega, \R)
$ the corresponding spaces of real-valued functions on a domain $\Omega\subset \R^n $.
If $\Omega$ is a bounded open set with Lipschitz boundary,
H\"older spaces are connected to Sobolev spaces by Sobolev's inequality which states for $ k - \nicefrac{n}{p} = r+\alpha $, $p\in(n,\infty]$, $ r < k $, $ \alpha \in (0,1] $ that 
\begin{equation}
        \label{eq:SobolevIneq}
        \tag{\arabic{section}.S}
	\|f\|_{\Hoeld{r}[\alpha]} \leq C \|f\|_{\Sobolev{k, p}}.
\end{equation}
with $ C = C(n, p) $.
From $\|\blank\|_\alpha \leq \|\blank\|_1 $ and the mean value theorem we get the elementary estimate
\begin{equation}\label{eq:Holder_elementaryEstimate}
        \|f\|_{\Hoeld{m}[\alpha]}
    \leq
        \|f\|_{\Hoeld{m+1}}.    
\end{equation}

\subsection{Mollification}

The tool for regularization will be mollification, i.e.\ convolution with a smooth function.
Convolution can be defined for any
    compactly supported function $ f \colon \R^n \to \R $ and any
    locally integrable function $ g \colon \R^n \to \R $ 
via
\begin{equation}\label{eq:convolution_def}
	f*g(x) \coloneqq \int f(x-h) g(h) \mathrm{d}h.	
\end{equation}
If one of the functions $f,g$ is multi-valued, i.e.\ $ \R^n \to \R^N$, the convolution is defined component-wise.
It is elementary to see that if in addition $f \in \Hoeld{m}$, then $f*g \in \Hoeld{k}$ for $ m = 0,1,\ldots $ \cite[\citeTheorem 1.3.1]{Hormander83}.
Moreover, in this case
\begin{equation}\label{eq:convolution_deriv}
	 \partial^I (f*g) = (\partial^I f)*g
\end{equation}
where $I$ is a multi-index f order $m$.
We will always assume that $ p, q \in [1, \infty] $ satisfy the relation $ 1 = \nicefrac{1}{p} + \nicefrac{1}{q} $.
The classical key tools will be H\"older's inequality
\begin{equation}
        \label{eq:HolderIneq}
        \tag{\arabic{section}.H}
	\|f g\|_{\Lebesgue{1}} \leq \|f\|_{\Lebesgue{p}} \|g\|_{\Lebesgue{q}}
\end{equation}
and Young's convolution inequality
\begin{equation}
        \label{eq:YoungIneq}
        \tag{\arabic{section}.Y}
	\|f * g \|_{\Lebesgue{\infty}} \leq \|f \|_{\Lebesgue{p}} \|g\|_{\Lebesgue{q}}.
\end{equation}

\begin{subequations}
For the definition of mollification, fix as the mollification kernel a smooth  function $\varphi\colon \R^n \to [0,1] $ supported on $ [-1,1]^n $ with $ \varphi(0) = \int \varphi(x) \mathrm{d} x = 1$. Set for $ t > 0 $
\begin{equation}\label{eq:smoothingKernels}
    \varphi_\flowVar(x) \coloneqq \flowVar^{-n}\varphi(\nicefrac{x}{\flowVar}).
\end{equation}
By substitution we have that $
        \int \varphi_t(x) \mathrm{d} x
    =   \int \varphi(x) \mathrm{d} x
    =   1
$.
Define the mollification operator by
\begin{equation}
        P_\flowVar
    \colon
        f \mapsto \varphi_\flowVar * f.
\end{equation}
As a first application of Young's convolution inequality in conjunction with \cref{eq:convolution_deriv} note that
\begin{equation}\label{eq:mollification bd}
        \| \nabla^m P_\flowVar f \|_{\Lebesgue{\infty}}
    \leq
        \| \nabla^m \varphi_t \|_{\Lebesgue{1}}
        \| f \|_{\Lebesgue{\infty}}.
\end{equation}
\end{subequations}

\subsection{Definition of chart Norms}\label{subsec:chart-norms}

H\"older classes of riemannian metrics allow to formulate celebrated regularity results in a more concise and little bit stronger fashion.
Let  $(M, g,\Basepoint)$ be a pointed $n$-dimensional riemannian manifold.
We introduce norms on charts, given by pointed maps
\[
    \psi \colon (\Openball{0}{r}, 0) \to (M,\Basepoint),
\]
i.e.\ $\psi(0) = \Basepoint$, where $\Openball{x}{r}$ denotes the open ball of radius $r$ around $x$ in the metric space to which $x$ belongs---here $0$ belongs to $\R^n$ with the euclidean metric.
In contrast, we will denote a closed ball by $ \Closedball{x}{r} $.
We will mainly use and adapt definitions from \cite[11.3.1-11.3.5]{Petersen16}.

\begin{dfn}\label{def:Holder_chartNorm}
    For a chart $\psi \colon (\Openball{0}{r}, 0) \to (M, g,\Basepoint)$ compatible with the smooth atlas of $M$ we define $\|\psi\|_{\Hoeld{m}[\alpha]}$,
    the \definiendum{chart norm of $\psi $ on the scale of $r$}
    as the minimal\footnote{Note that there are only finitely many defining conditions.} quantity $Q \geq 0$ for which the following conditions 
    are fulfilled
    \begin{enumerate}
        \item
            for the differentials we have the bounds $ |D\psi| \leq e^Q $ on $\Openball{0}{r}$ and $ |D\psi^{-1}| \leq e^Q $ on $ \psi(\Openball{0}{r}) $.
            Equivalently, this condition can be expressed in coordinates on $\psi$ by
            \begin{equation}\label{eq:normBd}\tag{\arabic{section}.N$_0$}
                e^{-2Q} \delta_{kl} v^k v^l \leq g_{kl} \leq e^{2Q} \delta_{kl} v^k v^l
            \end{equation}
            for every vector $v \in \R$.
        \item
            for the semi-norm from \cref{eq:Holder_seminorm}
            and any $ k = 0, 1, \ldots, m$
            \begin{equation}
                    \label{eq:normBd_Holder}
                    \tag{\arabic{section}.N$_{\Hoeld{m}[\alpha]}$}
                r^{k+\alpha} \|\nabla^k g_{\dbBlank}\|_\alpha  \leq Q
            \end{equation}
            where $ g_{\dbBlank} \coloneqq \psi^* g $.
    \end{enumerate}
    In this case we write
    \[
        \|\psi\|_{\Hoeld{m}[\alpha], r} \leq Q.
    \]
    We define $\|\psi\|_{\Hoeld{m}[\alpha]}^{\textnormal{harm}}$, the \definiendum{harmonic chart norm of $\psi $ on the scale of $r$}, by additionally imposing the condition
    \begin{enumerate}[resume]
        \item
            the chart $\psi$ is harmonic, meaning that each coordinate function $x_k$ ($k=1,\ldots,n$) is harmonic with respect to $g_\dbBlank$, i.e.\ the Laplace-Beltrami operator vanishes
            \begin{equation}
                    \label{eq:normBd_harmonic}
                    \tag{\arabic{section}.N$_{\textnormal{harm}}$}
                \Laplacian_{g_\dbBlank} x_k = 0.
            \end{equation}
    \end{enumerate}
    In this case we write
    \[
        \|\psi\|_{\Hoeld{m}[\alpha], r}^{\textnormal{harm}} \leq Q.
    \]
\end{dfn}

\begin{subequations}
We can directly extend this definition by
\begin{align*}
        \|(M,g,\Basepoint)\|_{\Hoeld{m}[\alpha], r}
    &=
        \inf_{\psi\colon (\Openball{0}{r},0) \to (M,\Basepoint) } \|\psi\|_{\Hoeld{m}[\alpha]},
\\
        \|(M,g)\|_{\Hoeld{m}[\alpha], r}
    &=
        \sup_{\Basepoint \in M } \|(M, g,\Basepoint)\|_{\Hoeld{m}[\alpha], r}
\end{align*}
\end{subequations}
(mutatis mutandis for $\|(M,g,\Basepoint)\|_{\Hoeld{m}[\alpha], r}^{\textnormal{harm}}$ and $ \|(M,g)\|_{\Hoeld{m}[\alpha], r}^{\textnormal{harm}} $).
In the same manner we can introduce the norm bounds on Sobolev scales
\begin{gather*}
        \|\psi\|_{\Sobolev{m}[p], r} \leq Q,\quad
        \|(M, g,\Basepoint)\|_{\Sobolev{m}[p], r} \leq Q
    ,\quad
        \|(M, g)\|_{\Sobolev{m}[p], r} \leq Q
\\
\left(
    \|\psi\|_{\Sobolev{m}[p], r}^{\textnormal{harm}} \leq Q,\quad
        \|(M, g,\Basepoint)\|_{\Sobolev{m}[p], r}^{\textnormal{harm}} \leq Q
    ,\quad
        \|(M, g)\|_{\Sobolev{m}[p], r}^{\textnormal{harm}} \leq Q
,\quad\text{resp.}
\right)
\end{gather*}
by retaining condition \cref{eq:normBd} (as well as \cref{eq:normBd_harmonic} if appropriate) and replacing condition \cref{eq:normBd_Holder} by
\begin{equation}
    \label{eq:normBd_Sobolev}\tag{\arabic{section}.N$_{\Sobolev{m}[p]}$}
        r^{k - n/p} \|\nabla^k g_\dbBlank \|_{\Lebesgue{p}}
    \leq
        Q
\text{.}
\end{equation}
for $k=0,1,\ldots, m$.
Note that $
    \|(M,g,\Basepoint)\|_{\Hoeld{m}[\alpha], r}
$ and $
    \|(M,g,\Basepoint)\|_{\Hoeld{m}[\alpha], r}^{\textnormal{harm}}
$ are realized by charts by an application of the Arzel\`a–Ascoli theorem.
(The same holds for $
    \|(M,g,\Basepoint)\|_{\Sobolev{m}[p], r}
$ and $
    \|(M,g,\Basepoint)\|_{\Sobolev{m}[p], r}^{\textnormal{harm}}
$ as a consequence of the Banach–Alaoglu theorem.)
Finally, let 
\begin{gather*}
    \mathcal{M}^n(\Hoeld{m}[\alpha]\leq_r Q)
,\quad
    \mathcal{M}^n(\Sobolev{m}[p]\leq_r Q)
\\
\left(
    \mathcal{M}^n(\Hoeld{m}[\alpha]\leq_r^{\textnormal{harm}} Q)
,\quad
    \mathcal{M}^n(\Sobolev{m}[p]\leq_r^{\textnormal{harm}} Q)
, \quad\text{resp.}
\right)
\end{gather*}
denote the space of isomorphism class of $n$-dimensional, pointed, smooth riemannian manifolds $ (M, g,\Basepoint) $ with
\begin{gather*}
   \|(M, g)\|_{\Hoeld{m}[\alpha], r} \leq Q
    ,\quad
    \|(M, g)\|_{\Sobolev{m}[p], r} \leq Q
, \quad\text{resp.}
\\
\left(
    \|(M, g)\|_{\Hoeld{m}[\alpha], r}^{\textnormal{harm}} \leq Q
,\quad
    \|(M, g)\|_{\Sobolev{m}[p], r}^{\textnormal{harm}} \leq Q
, \quad\text{resp.}
\right).
\end{gather*}
These spaces are endowed with the Gromov-Hausdorff topology.
Note the elementary estimate \cite[\citeProposition 11.3.2 (4)]{Petersen16}
\begin{multline}
        \label{eq:lengthEstimate}
        \tag{\arabic{section}.D}
        \cref{eq:normBd}
    \implies\\
            e^{-Q}\min\{ \dist{x}{y}, 2r - \dist{0}{x} \}
        \leq
            \dist{\psi(x)}{\psi(y)}_g
        \leq
            e^Q \dist{x}{y}
        \leq
            e^Q \dist{x}{y}
\end{multline}
for all $ x, y \in \Openball{0}{r}$ and $|\dbBlank|$ the euclidean norm.

Having introduced spaces with a global bound on the metric tensor in local coordinates, one may be inclined to ask why we did not assume any regularity assumption on changes of coordinates.
The answer is found in Schauder estimates, standard estimates on the regularity of solutions of elliptic PDEs.
The crucial fact can be stated as follows \cite[Problem~6.1~(a)]{GilbargTrudinger15}:
On a bounded open set $\Omega$
let $u\colon \Omega \to \R $ be a $ \Hoeld{m+2}[\alpha] $-solution ($ m \geq 0 $) of $ 
    (a^{ij}(x)\partial_i\partial_j + b^i(x) \partial_i + c(x) )u = f 
$ (summation convention) and assume that the coefficients of $L$ satisfy $ a^{ij}\xi_i\xi_j \geq \lambda |\xi|^2 $ and $ 
        \| \nabla^m a \|_\alpha, \| \nabla^m b \|_\alpha, \| \nabla^m c \|_\alpha
    \leq
        \Lambda
$. 
If $\Omega' \subset \Omega$ with $\overline{\Omega'} \subsetneqq \Omega$, then
\[
        \| u\|_{\Hoeld{m+2}[\alpha]}
    \leq
        C( \|u\|_{\Hoeld{0}} + \| f \|_{\Hoeld{m}[\alpha]} )
\]
on $\Omega'$ with $ C = C(n, m, \alpha, \lambda, \Lambda, \dist{\Omega'}{\setBd \Omega}_{\mathrm{H}} ) $
where $ \dist{\Omega'}{\setBd \Omega}_{\mathrm{H}} $ denotes the Hausdorff distance.

If we apply this statement to a transition function $\psi_i^{-1}\circ \psi_j$ for two charts $\psi_i$, $\psi_j $ with $ 
        \|\psi_i\|_{\Hoeld{m}[\alpha], r}^{\textnormal{harm}},
        \|\psi_j\|_{\Hoeld{m}[\alpha], r}^{\textnormal{harm}}
    \leq
        Q
$ we first notice that by harmonicity $ 
        \| \Laplacian_g \psi_i^{-1}\circ \psi_j \|_{\Hoeld{m}[\alpha]}
    =
        \| 0 \|_{\Hoeld{m}[\alpha]}
    =
        0
$ and moreover, in harmonic coordinates, as calculated e.g.\ in \cite[11.2.3]{Petersen16}, $
    \Laplacian_g = g^{ij} \partial_i \partial_j
$. If we restrict to the domain $ 
        \Omega'
    \coloneqq
        \psi_j^{-1}(\psi_i(\Openball{0}{\nicefrac{r}{2}}) \cap \psi_j(\Openball{0}{\nicefrac{r}{2}}) )
$ the above estimate becomes
\begin{equation}
        \label{eq:transition_estimate}
        \tag{\arabic{section}.Sch}
        \| \psi_i^{-1}\circ \psi_j\|_{\Hoeld{m+2}[\alpha]}
    \leq
        C \|\psi_i^{-1}\circ \psi_j\|_{\Hoeld{0}}
    \leq
        C \cdot \nicefrac{r}{2}
\end{equation}
where $ C = C( n, m, \alpha, Q, r) $.
For non-harmonic chart norms a similar result with one lower degree of regularity can be found in \textcite{Taylor06}.

\subsection{Precompactness theorems for riemannian manifolds}\label{subsec:precptness-thms}
We review some fundamental results.
For every $ n \geq 2$ and $s, Q > 0 $ 
\begin{equation}\label{eq:fundamentalThm}\tag{\arabic{section}.K$_{\Hoeld{m}[\alpha]}$}
    \mathcal{M}^n(\Hoeld{m}[\alpha]\leq_r Q)
    \text{ is compact}
\end{equation}
with respect to the Gromov-Hausdorff topology.
This statement is sometimes called Fundamental Theorem of Convergence Theory \cite[11.3.5]{Petersen16}.

Let $\iota > 0$, $\alpha \in (0,1)$, $ p \in (n,\infty) $, and $ \underline{\kappa} \in \R $.
For all $Q>0$ there is $ r> 0 $ such that every pointed riemannian manifold $ (M, g,\Basepoint) $ satisfies
\begin{multline}\label{eq:lowerRicciBd}\tag{\arabic{section}.K$_{\Ricci}$}
    \text{injectivity radius$(M) \geq \iota$ and $\Ricci \geq - \underline{\kappa}^2 $}
        \\
        \implies
            (M, g,\Basepoint) \in
            \mathcal{M}^n(\Hoeld{\alpha}\leq_r^{\textnormal{harm}} Q) \cap \mathcal{M}^n(\Sobolev{1}[p]\leq_r^{\textnormal{harm}} Q);
\end{multline}
see \cite{AndersonCheeger92}.
Likewise, we obtain a result that is stronger by one derivative for an absolute Ricci bound:
For every $ \kappa \geq 0 $, $\iota > 0$, $\alpha \in (0,1)$, $p\in(n,\infty)$, and $Q\in (0, \log(2)) $ there is an $r > 0$ such that any smooth manifold $(M, g,\Basepoint)$ 
\begin{multline}\label{eq:absoluteRicciBd}\tag{\arabic{section}.K$_{|\Ricci|}$}
    \text{injectivity radius$(M) \geq \iota$ and $\|\Ricci M\|_{\Lebesgue{\infty}} \leq \kappa$}
        \\
        \implies
            (M, g,\Basepoint) \in
            \mathcal{M}^n(\Hoeld{1}[\alpha]\leq_r^{\textnormal{harm}} Q) \cap \mathcal{M}^n(\Sobolev{2}[p]\leq_r^{\textnormal{harm}} Q);
\end{multline}
the result appeared several times, e.g. \cite{Anderson90} and \cite[\citeTheorem 11]{HebeyHerzlich97},
    and found its ways into textbooks \cite[11.4.1]{Petersen16}.
Note that the H\"older parts of \cref{eq:lowerRicciBd} and \cref{eq:absoluteRicciBd} are implied by their Sobolev parts via So\-bo\-lev's inequality \cref{eq:SobolevIneq}.

\begin{bem}
    Actually, these regularity results can be improved in at least two ways that we do not use but would like to emphasize.
    The exponent $\alpha $ can be improved from $\alpha \in (0,1)$ to 1
    by replacing the H\"older scales with H\"older-Zygmund scales $\Zygmund{s}$, $ s > 0 $.
    The space $\Zygmund{s}(\Omega)$ coincides with $\Hoeld{m}[\alpha]$ if $s = m + \alpha$ and $\alpha \neq 0,1$.
    Otherwise $ \Hoeld{m} \subset \Zygmund{m}(\omega) $.
    See \cite{Taylor07} for the case of lower Ricci bounds and \cite{AndersonTaylorEtAl04} for the case of absolute Ricci bounds.
    The norm bound in case of an absolute bound on Ricci curvature can be generalized to the case of manifolds with boundary \cite{AndersonTaylorEtAl04}.
\end{bem}

\section{Mollified riemannian curvature tensor}\label{sec:mollified-curvature-tensor}

In this section we begin by proving a local version of \cref{thm:mollifiedCurv convexCombi}. Recall that $g_\dbBlank = \psi^* g$ and $P_\flowVar(g_\dbBlank) = \varphi_\flowVar * g_\dbBlank $.
The results will be formulated using the Sobolev chart norm $\|\blank\|_{\Sobolev{2}[p], r}$ from \cref{eq:normBd_Sobolev}.

\begin{prop}\label{prop:mollifiedCurv}
    Let $p > 2n $, $r, Q> 0$ and choose $\beta \in (\nicefrac{1}{2}, 1 - \nicefrac{n}{p})$.
    There exists some $\flowVarMax \in (0, \nicefrac{r}{2}) $ such that for a pointed smooth riemannian manifold $(M, g,\Basepoint)$, a chart $\psi\colon (\Openball{0}{r}, 0) \to (M,\Basepoint) $,
    and any section $\vec{v} \in \Gamma \Tangent^{3,1} \Openball{0}{r} $
    we have
    \[
            \|\psi\|_{\Sobolev{2}[p], r} \leq Q
        \implies
                \Riem_{P_\flowVar (g_\dbBlank)}(\psi(x)) (\vec{v})
            - 
                \sup_{y\in \Openball{x}{\flowVar}} \Riem_{g_\dbBlank}(y) (\vec{v})
            \leq 
                C \|\vec{v}\|_{\Lebesgue{\infty}}^4 \flowVar^\beta
    \]
    for any $ \flowVar \in (0,\flowVarMax] $, $x \in \Openball{0}{\nicefrac{r}{2}} $, a constant $
        C = C(n,p,r,Q, \beta )
    $, and considering $\vec{v}$ as a function valued in $\R^{4n}$ by the canonical euclidean identification.
\end{prop}

The following generalization of the above proposition to a convex combination of mollified metrics will be of even greater interest.
Recall that the $g$-norm of a section $v$ in a tangent bundle is defined by $
    \|v\|_g \coloneqq \sup_{x\in M} \|v(x)\|_{g(x)}
    $.
For a cotangent vector $ v_x \in \Tangent^*_x M $ the $g$-norm $\|v_x\|_{g(x)}$ is defined as \linebreak[3] $\|(w\mapsto \langle w,\blank\rangle)^{-1} (v_x)\|_{g(x)}$.
The $g$-norm of a section $ v \in \Gamma \Tangent^* M $ is $\|v\|_g \linebreak[1] \coloneqq \sup_{x\in M}\linebreak[1] \|v(x)\|_{g(x)}$.
Finally, for a section in a product of bundles $
        \vec{v} 
    = 
        (v_1, \ldots, v_k,\linebreak[1] v_{k+1}, \ldots, v_{k+l}) 
    \in 
        \Gamma( 
            \Tangent M \times \ldots \times \Tangent M \times 
            \Tangent^* M \times \ldots \times \Tangent^* M  
        )
$ we define $
        \|\vec{v}\|_g 
    = 
        \sup\limits_{x\in M}
        \|v_1(x)\|_{g(x)} \cdots \|v_k1(x)\|_{g(x)} \cdot \|v_{k+1}1(x)\|_{g(x)} \cdots \|v_{k+l}1(x)\|_{g(x)}
$.

\begin{dfn}\label{def:mollify_g}
    \begin{subequations}
    Let $r > 0$ and $ M $ be a smooth manifold, which is covered by a locally finite family of charts $ \{\psi_i\}_I $ with a \definiendum{corresponding partition of unity} $ \{\varrho_i\}_I $, i.e.\ 
    \begin{equation}\label{eq:def:mollify_g charts_weights}
        \left\{\psi_i\colon \Openball{0}{r} \to M\right\}_{i\in I}
    ,\quad
        \left\{\varrho_i\colon M \to [0,1]\right\}_{i\in I}
    \end{equation}
    such that for each $i\in I$ the function $\varrho_i$ is smooth and
    \begin{equation}\label{eq:def:mollify_g charts_partition}
        \supp \varrho_i \subset \psi_i(\Openball{0}{\nicefrac{3r}{4}})
    ,\quad
        \sum_{j\in I} \varrho_j\Bigg|{}_{\bigcup\limits_{j\in I} \psi_j(\Openball{0}{\nicefrac{r}{2}}) }  = 1
    .
    \end{equation}
    Define for $\flowVar\in (0,\nicefrac{r}{2})$ the \definiendum{mollified metrics}
	\begin{align}\label{eq:def:mollify_g}
			 g^{ [\flowVar] } 
		&\coloneqq 
		    P_{\flowVar, \{\psi_i\}_I, \{\varrho_i\}_I } (g)
		\coloneqq 
			\sum_{i\in I} \varrho_i g^{[\flowVar, \psi_i ]}
	\\
	\intertext{on $ \bigcup\limits_{j\in I} \psi_j(\Openball{0}{\nicefrac{r}{2}}) $ where}
	        g^{ [\flowVar,  \psi_i ] }
	   &\coloneqq
	        (\psi_i^{-1})^* P_\flowVar ( \psi_i^* g)
	.
	\end{align}
    \end{subequations}
\end{dfn}

\ThmMollifiedCurv*

\begin{subequations}
Throughout the proof, whose ansatz was pointed out in \cref{eq:ansatz}, we will consider a metric tensor $g_\dbBlank = \psi^* g$
with 
\begin{equation}\label{eq:cond_psi}
    \|\psi\|_{\Sobolev{2}[p], r}, \|\psi\|_{\Hoeld{1}[\alpha], r} \leq Q  
\end{equation}
such that
\begin{equation}\label{eq:cond_n_p_q}
		\frac{1}{p} + \frac{1}{q} = 1,
	\quad
		n < p < \infty,
	\quad
		\frac{n}{p} < \alpha < 1,
	\quad\text{and}\quad
	    \beta \leq \alpha-\nicefrac{n}{p}.
\end{equation}
These conditions imply $ p,q \in (0,\infty)$ and $ \alpha\in (0,1) $.
To guarantee this condition without loss of generality, choose $ 
    \alpha \in (1 - \nicefrac{2n}{p}, 1 - \nicefrac{n}{p} )
$, apply Sobolev's inequality \cref{eq:SobolevIneq}, and replace $Q$ by $\max\{Q, C Q\}$, where $C$ is the constant from Sobolev's inequality.
By $g^\dbBlank$ we denote the inverse of $g_\dbBlank$.
\end{subequations}

\begin{lemma}\label{lem:curvature_decomposition}
    In local coordinates the riemannian curvature tensor $\Riem$ of a smooth manifold $(M, g)$ can be written using notation \cref{eq:nabla_leq} as
    \[
			R^{\,.}{}_{\triBlank} 
		=
			A(g^\dbBlank, \nabla^2 g_\dbBlank)
		  +
		    B(\nabla g_\dbBlank, \nabla^{\leq 1} g^\dbBlank)
	\]
	where
	\begin{itemize}
	    \item $  A $ is bilinear;
	    \item $ B $ is a polynomial of degree not bigger than 4 and without constant terms; and
	    \item the coefficients of $A$ and $B$ depend only on the dimension.
	\end{itemize}
\end{lemma}

\begin{proof}
	The standard coordinate definition the riemannian curvature tensor has the form \cite[C3]{Taylor:2}
	\begin{align*}
			\Gamma^{i}_{kl} 
		&=
			\frac{1}{2}g^{im}(g_{mk,l}+g_{ml,k}-g_{kl,m})
	\\
			R^{\varrho }{}_{\sigma \mu \nu }
		&=
			\partial_\mu\Gamma^\varrho_{\nu\sigma} - \partial_\nu\Gamma^\varrho_{\mu\sigma} + \Gamma^\varrho_{\mu\lambda}\Gamma^\lambda_{\nu\sigma} - \Gamma^\varrho_{\nu\lambda}\Gamma^\lambda_{\mu\sigma}
	\end{align*}
	where $ f_{i,j} = \partial_j f_i = \frac{\partial}{\partial x_j} f_i$.
	From this we read that
	\[
			R^{\,.}{}_{\triBlank}
		= 
			  L_2( g^\dbBlank, \nabla^2 g_\dbBlank ) 
			+ L_1( \nabla g^\dbBlank, g_\dbBlank) 
			+ Q_1( g^\dbBlank, \nabla g_\dbBlank)
	\]
	where the coefficients of $L_2$, $ L_1 $, and $ Q_1 $ depend only on the dimension and
	\begin{itemize}
		\item 
			$ L_2 $ is linear in $ \nabla^2 g_{\dbBlank} $ and $ g^{\dbBlank} $,
		\item 
			$ L_1 $ is linear in $ \nabla g^{\dbBlank} $ as well as in $ \nabla g_{\dbBlank} $,
		\item 
			$ Q_1 $ is quadratic in both $ g^{\dbBlank} $ and $ \nabla g_{\dbBlank} $.
	\end{itemize}
	Obviously, choose $ A = L_2 $ and $ B = L_1 + Q_1 $.
\end{proof}

\begin{lemma}\label{lem:polyApprox}
    Let $ f, g \in \Lebesgue{\infty}(\Omega, \R^N) $ and $p\in \R[X_1,\ldots, X_N]$ a polynomial without constant term.
    Then
    \[
            \|p(f) - p(g)\|_{\Lebesgue{\infty}} 
        \leq 
            C \|f-g\|_{\Lebesgue{\infty}}
    \]
    where $C = C(p, \|f\|_{\Lebesgue{\infty}}, \|g\|_{\Lebesgue{\infty}})$.
\end{lemma}

\begin{proof}
    By the triangle inequality, it is sufficient to prove the claim for a monomial $p = X_{i_1}\cdots X_{i_d} $.
    Then the claim follows from a telescope argument
    \begin{multline*}
            \|
                f_{i_1}\cdots f_{i_d} - g_{i_1}\cdots g_{i_d}
            \|_{\Lebesgue{\infty}} 
        \\\leq
              \| 
                (f_{i_1}-g_{i_1})f_{i_2}\cdots f_{i_d}
              \|_{\Lebesgue{\infty}} 
            +
              \|
                g_{i_1} (f_{i_2}-g_{i_2})\cdots f_{i_d}
              \|_{\Lebesgue{\infty}} 
            +
              \ldots\\\ldots
            +
              \|
                g_{i_1}\cdots (f_{i_d} - g_{i_d})
              \|_{\Lebesgue{\infty}}
    \end{multline*}
    since each term $g_{i_1}\ldots (f_{i_k}-g_{i_k})\ldots f_{i_d}$ is $\Lebesgue{\infty}$-bounded by $ 
        \|f\|_{\Lebesgue{\infty}}^{k-1} \cdot \|g\|_{\Lebesgue{\infty}}^{n-k-1} \cdot
        \|f-g\|_{\Lebesgue{\infty}}
    $ (for $k=1,\ldots,d$).
\end{proof}

We now turn to the crucial $\Lebesgue{p}$-estimate.
It is a consequence of H\"older's inequality \cref{eq:HolderIneq} for indices 1 and $\infty$ stating that for any $ q \in (1,\infty) $
\begin{align} 
		\|\varphi_\flowVar\|_{\Lebesgue{q}}
	&=
		\| \flowVar^{-nq} \chi_{[-\flowVar, \flowVar]^{\times n}}  (\varphi(h/\flowVar))^q \|_{\Lebesgue{1}}^{1/q}
\nonumber
\\
	&\leq
		\| \flowVar^{-nq} \chi_{[-\flowVar, \flowVar]^{\times n}} \|_{\Lebesgue{1}}^{1/q}
		\| (\varphi(h/\flowVar))^q \|_{\Lebesgue{\infty}}^{1/q}
\nonumber
\\
	&=
		\flowVar^{-n}\|  \chi_{[-\flowVar, \flowVar]^{\times n}} \|_{\Lebesgue{1}}^{1/p}
		\| (\varphi(h/\flowVar))^q \|_{\Lebesgue{\infty}}^{1/q}
\nonumber
\\
	&=
		\flowVar^{-n} (2 \flowVar)^{n \cdot 1/q} \cdot 1^{1/q}
\nonumber
\\
    &=
        2^{n/q} \flowVar^{- n(1 - 1/q)}
\nonumber
\\
\label{eq:mollifier_estimate}
	&=
		2^{n/q} \flowVar^{- n/p}.
\end{align}

Recall the definitions of the commutator of two operators $P, Q$
\[
        [P, Q]
    \colon
        f \mapsto PQ - QP
\]
and the multiplication operator
\[
        T_{g}
    \colon
        f \mapsto g f.
\]

Finally, note that condition \cref{eq:normBd} implies that the eigenvalues of $ g_\dbBlank $ are in $[e^{-2Q}, e^{2Q}]$.
This implies
\begin{equation}\label{eq:detBd}
    \det(g_\dbBlank) \in [e^{-2Qn}, e^{2Qn}].
\end{equation}

\begin{lemma}[Commutator-like estimate]\label{lem:commutatorLikeEstimate}
	Let $ n, p, q, \alpha $ satisfy \cref{eq:cond_n_p_q} and $\Omega \subset \R^n$ be an open domain.
	Let $ f\in \Lebesgue{p}(\R^n)  $ and $ a_{(\blank)}(\blank)\colon \Omega \times [0,\flowVarMax] \to \R $ such that for some constant $C_a$ we have
	\[
	    \|a_\flowVar\|_\alpha \leq C_a
    \quad\text{and}\quad
        \| a_\flowVar - a_0 \|_{ \Hoeld{0}} \leq C_a\flowVar^\alpha
	\]
	for all $ \flowVar \in [0,\flowVarMax] $.
	Then 
	\[
			\| P_\flowVar T_{a_0} (f) - T_{a_\flowVar} P_\flowVar (f) \|_{\Lebesgue{\infty}} 
		\leq 
			2^{(n+1)/q} C_a \| f\|_{\Lebesgue{p}} \flowVar^{\alpha - n/p }.
	\]
\end{lemma}

\begin{proof}
    First note that by Young's convolution inequality \cref{eq:YoungIneq} and \cref{eq:mollifier_estimate}
    \begin{align*}
            \|[P_\flowVar, T_g ] f\|_{\Lebesgue{\infty}}
        &=
           \esssup_x \left|\int ( g(x-h) - g(x) ) f(x - h) \varphi_\flowVar (h) \mathrm{d} h\right|
    \nonumber\\
        &\leq
            \esssup_x \int | g(x-h) - g(x) | \cdot | f(x - h) | \varphi_\flowVar (h) \mathrm{d} h
    \nonumber\\
        &\leq
            \esssup_x \int \|g\|_\alpha |h|^\alpha | f(x - h) | \varphi_\flowVar (h)  \mathrm{d} h
    \nonumber\\
        &\leq
            \|g\|_\alpha \flowVar^\alpha \| P_\flowVar (|f|) \|_{\Lebesgue{\infty}} 
    \nonumber\\
        &=
            \|g\|_\alpha \flowVar^\alpha \| f \|_{\Lebesgue{p}} \|\varphi_\flowVar \|_{\Lebesgue{q}}
    \nonumber\\
    \label{eq:commutator estimate}
        &\leq
             2^{n/q} \flowVar^{\alpha - n/p} \|g\|_\alpha \| f \|_{\Lebesgue{p}}
    \text{.}
    \end{align*}
    Using this estimate as well as \cref{eq:YoungIneq,eq:mollifier_estimate} again we conclude the proof:
	\begin{align*}
	\MoveEqLeft[2] 
	    \| P_\flowVar T_{a_0} (f) - T_{a_\flowVar} P_\flowVar (f) \|_{\Lebesgue{\infty}} 
	\\
		&\leq
				\| [ P_\flowVar , T_{a_\flowVar} ] f \|_{\Lebesgue{\infty}} 
			+
				\| P_\flowVar T_{a_0} (f) - P_\flowVar T_{a_\flowVar} f \|_{\Lebesgue{\infty}}
	\\
	    &=
	            \| [ P_\flowVar , T_{a_\flowVar} ] f \|_{\Lebesgue{\infty}} 
	        +
	            \| P_\flowVar T_{a_0 - a_\flowVar} (f) \|_{\Lebesgue{\infty}}
	\\
		&\leq
		        \| [ P_\flowVar , T_{a_\flowVar} ] f \|_{\Lebesgue{\infty}} 
	        +
	            \| \varphi_t \|_{\Lebesgue{q}}
	            \| T_{a_0 - a_\flowVar} (f) \|_{\Lebesgue{p}}
	\\
		&\leq
				2^{n/q} \flowVar^{\alpha - n/p} \|a_\flowVar\|_\alpha \|f\|_{\Lebesgue{p}}
			+
			    2^{n/q} \flowVar^{- n/p} \|a_0 - a_\flowVar\|_{\Hoeld{0}} \| f\|_{\Lebesgue{p}}
	\\
		&\leq
				C_a \flowVar^\alpha \|f\|_{\Lebesgue{p}} 2^{n/q} \flowVar^{- n/p}
			+
				C_a \flowVar^\alpha \| f\|_{\Lebesgue{p}} 2^{n/q} \flowVar^{- n/p}
	\\
		&=
			2^{(n+1)/q} C_a \| f\|_{\Lebesgue{p}} \flowVar^{\alpha - n/p }.
	\qedhere
\	\end{align*}
\end{proof}

\begin{lemma}\label{lem:HolderEstimate}
    Let $ r> 0 $, $m = 0, 1, \ldots $, and $\alpha \in (0,1] $.
    Given a function $f \in \Hoeld{m}[\alpha](\Openball{0}{r}, \R) $ we have
    \[
            \| f - P_\flowVar f \|_{\Hoeld{m}}
        \leq
            \flowVar^\alpha \|f\|_{\Hoeld{m}[\alpha] }
    \]
    on $\Openball{0}{\nicefrac{r}{2}}$ for $ \flowVar \in (0,\nicefrac{r}{2})$.
\end{lemma}

\begin{proof}
    Let $k = 0, \ldots, m $.
    The lemma follows from
    \begin{align*}
            \|\nabla^k (P_\flowVar(f) - f) \|_{\Hoeld{0}}
        &=
            \| P_\flowVar(\nabla^k f) - \nabla^k f \|_{\Hoeld{0}}
    \\
        &=
            \sup_{x \in \Openball{0}{\nicefrac{r}{2}}}
                \left|
                    P_\flowVar(\nabla^k f)- \int \varphi_\flowVar (h) \nabla^k f (x) \mathrm{d} h
                \right|
    \\
        &=
            \sup_{x \in \Openball{0}{\nicefrac{r}{2}}}
                \left|
                    \int \varphi_\flowVar (h) ( \nabla^k f (x-h) - \nabla^k f (x)) \mathrm{d} h 
                \right|
    \\
        &\leq
            \sup_{x \in \Openball{0}{\nicefrac{r}{2}}}
                \int |\varphi_\flowVar (h)| \cdot |h|^\alpha \|\nabla^k f\|_{\alpha} 
                \mathrm{d} h
    \\
        &\leq
             \| \nabla^k f \|_{\alpha} \int |\varphi_\flowVar (h)| \flowVar^\alpha \mathrm{d} h
    \\
        &=
            \flowVar^\alpha\| \nabla^k f \|_{\alpha}.
    \qedhere
    \end{align*}
\end{proof}

\begin{proof}[Proof of \cref{prop:mollifiedCurv}]
    \begin{subequations}
    As discussed above at \cref{eq:cond_psi} we can assume without loss of generality $
        \|\psi\|_{\Sobolev{2}[p], r}, \|\psi\|_{\Hoeld{1}[\alpha], r} \leq Q
    $ with properties \cref{eq:cond_n_p_q}.
    By assumption \cref{eq:normBd_Holder} and \cref{lem:HolderEstimate} this implies 
    \begin{equation}\label{eq:prop:mollifiedCurv mollifyBd}
            \|P_\flowVar(g_\dbBlank) - g_\dbBlank \|_{\Hoeld{1}}
        \leq
            C_{n, r} Q \flowVar^\alpha
    \end{equation}
    for $\flowVar \in (0, \nicefrac{r}{2}) $ and
    for a dimension- and radius-depending constant $C_{n,r}$.
    This in conjunction with the bounds \cref{eq:detBd} implies
    \begin{equation}\label{eq:mollifiedCurv_detBd}
        \det(g_\dbBlank), \det(P_\flowVar(g_\dbBlank)) \in \left[e^{-2Qn-1}, e^{2Qn+1}\right]
    \end{equation}
    for all $t \in (0, \flowVarMax] $ when $\flowVarMax$ is chosen  sufficiently small.
    This implies that
    \begin{equation}\label{eq:mollification_inverseBd}
        \| P_\flowVar(g_\dbBlank)^{-1} \|_{\Hoeld{1}[\alpha]}, 
        \| g^\dbBlank \|_{\Hoeld{1}[\alpha]}
        \leq C_{Q, r} 
    \end{equation}
    for a constant $ C_{Q, r} $ \cite[\citeCorollary 16.30]{Dacorogna11}.
    Let $\mathrm{M}_{kl}$ denote the $(k,l)$-minor of $g_\dbBlank$, $\mathrm{M}_{kl}^{\flowVar}$ the $(k,l)$-minor of $P_\flowVar(g_\dbBlank)$, and 
    $\mathrm{M}_\dbBlank, \mathrm{M}_\dbBlank^\flowVar$ the corresponding matrices of minors.
    Lastly, we need an estimate on the difference of the inverse of the mollified metric:
    \begin{align}
            \MoveEqLeft[2]\|(P_\flowVar(g_\dbBlank))^{-1} - g_\dbBlank^{-1} \|_{\Hoeld{1}}
    \nonumber
    \\
        &\leq
            \left\| 
                \frac{1}{\det P_\flowVar (g_\dbBlank)} \mathrm{M}_\dbBlank^\flowVar
                - \frac{1}{\det (g_\dbBlank)} \mathrm{M}_\dbBlank
            \right\|_{\Hoeld{1}}
    \nonumber
    \\
        &\leq
                \left|
                    \frac{1}{\det P_\flowVar (g_\dbBlank)} 
                    - \frac{1}{\det (g_\dbBlank)}
                \right|
                \left\|  \mathrm{M}_\dbBlank^\flowVar \right\|_{\Hoeld{1}}
            + 
                \frac{1}{\det (g_\dbBlank)}
                \left\| \mathrm{M}_\dbBlank^\flowVar - \mathrm{M}_\dbBlank \right\|_{\Hoeld{1}}
    \nonumber
    \\
    \intertext{proceeding by using \cref{eq:mollifiedCurv_detBd} twice---directly and by the standard estimate $ 
            \left| \frac{1}{\xi} - \frac{1}{\eta} \right| 
        \leq 
            \left|\left(\frac{\mathrm{d} }{\mathrm{d} x} \frac{1}{x}\right)(x_0)\right|\linebreak[0]
            |\xi - \eta| 
        =
            |x_0|^{-2} |\xi - \eta|
    $ for $\xi, \eta \geq x_0  $ 
    }
        &\leq
              \frac
                {|\det P_\flowVar (g_\dbBlank) - \det (g_\dbBlank) | }
                {\left( e^{-2Qn - 1} \right)^2}
              \left\| \mathrm{M}_\dbBlank^\flowVar \right\|_{\Hoeld{1}}
            +
                e^{2Qn+1}
                \left\| \mathrm{M}_\dbBlank^\flowVar - \mathrm{M}_\dbBlank \right\|_{\Hoeld{1}}
    \nonumber
    \\
        &\leq
              e^{4Qn + 2} \cdot C_{n,r} \cdot \flowVar^\alpha Q \cdot C_n
            +
              e^{2 Q n + 1} \cdot C_{n,r}' \cdot \flowVar^\alpha Q
    \nonumber
    \\
    \label{eq:mollification_inverseDiffBd}
        &\leq
            C_{n,r,\alpha, Q} \flowVar^\alpha.
    \end{align}
    where in the penultimate step we used \cref{lem:HolderEstimate,eq:prop:mollifiedCurv mollifyBd} as well as standard product estimates \cite[\citeTheorem 16.28]{Dacorogna11} providing suitable constants $C_n$ and $C_{n,r}'$.
	\end{subequations}
	
	Finally, we can apply the ansatz proposed in \cref{eq:ansatz}:
 	\begin{align}
 	\nonumber
 	        \Riem_{P_\flowVar (g_\dbBlank)}(x)(\vec{v})
	    &\leq
	           P_\flowVar \left(\Riem^{\,.}_{\triBlank}(x)\right)(\vec{v})
	        +
	            \left|
	                P_\flowVar (\Riem^{\,.}_{\triBlank})(x)(\vec{v}) - \Riem_{P_\flowVar (g_\dbBlank)}(x)(\vec{v})
	            \right|
	\\
	\label{eq:ansatz_restate}
	    &\leq
	            \sup_{y\in\Openball{x}{\flowVar}}  
    	            \Riem_{g_\dbBlank}(y) (\vec{v})    
    	   +
    	    \left\|
                P_\flowVar (\Riem^{\,.}_{\triBlank}) - \Riem_{P_\flowVar (g_\dbBlank)}
            \right\|_{\Lebesgue{\infty}} 
            \|\vec{v}\|_{\Lebesgue{\infty}}^4
    \text{.}
 	\end{align}
 	It remains to show that $
 	    \|
              P_\flowVar (\Riem^{\,.}_{\triBlank}) 
            - \Riem_{P_\flowVar (g_\dbBlank)}
        \|_{\Lebesgue{\infty}} 
	$ vanishes with modulus $\flowVar^\alpha$ as $\flowVar\to 0$.
	Set $
	        b
	    \coloneqq 
	        B(\nabla g_\dbBlank, \nabla^{\leq 1} g^\dbBlank)
	$ and $
	        b_{\flowVar} 
	    \coloneqq 
	        B(\nabla P_\flowVar g_\dbBlank, \nabla^{\leq 1} (P_\flowVar g_\dbBlank)^{-1} )
	$.
	Observe
	\begin{align*}
    \MoveEqLeft
    	    \|
                P_\flowVar (\Riem^{\,.}_{\triBlank}) - \Riem_{P_\flowVar (g_\dbBlank)}
            \|_{\Lebesgue{\infty}} 
    \\
        &\leq    
            \left\|
                P_\flowVar A\left(g^\dbBlank, \nabla^2 g_\dbBlank \right)
            -
                A\left( (P_\flowVar g_\dbBlank)^{-1}, 
                  \nabla^2 P_\flowVar g_\dbBlank\right)
            \right\|_{\Lebesgue{\infty}} 
          +
            \left\| 
                P_\flowVar( b ) -b_{\flowVar} 
            \right\|_{\Lebesgue{\infty}} 
    \\
        &\leq
            \begin{multlined}[t]
                \left\|
                    P_\flowVar A\left(g^\dbBlank, \nabla^2 g_\dbBlank \right)
                -
                    A\left( (P_\flowVar g_\dbBlank)^{-1}, 
                      P_\flowVar \nabla^2 g_\dbBlank\right)
                \right\|_{\Lebesgue{\infty}} 
              \\+
                \left\| P_\flowVar( b ) - b \right\|_{\Lebesgue{\infty}}  
              +
                \left\| b - b_{\flowVar} \right\|_{\Lebesgue{\infty}}  
            \end{multlined}
	\end{align*}
    We will give estimates for each summand:
    \begin{itemize}
        \item 
            For the first summand express $A$ as $A(x,y) = \sum_{i,j} a_{ij} x_i y_j $.
            By linearity of convolution $ 
              P_\flowVar (A(x,y)) = \sum_{i,j} a_{ij} P_\flowVar T_{x_i} (y_j)
            $.
            On the other hand $ 
              A(x, P_\flowVar y) = \sum_{i,j} a_{ij} T_{x_i} P_\flowVar (y_j) 
            $. This is to say that the difference of both expressions is $
                \sum_{i,j} a_{ij} [P_\flowVar, T_{x_i}](y_j)
            $.
            Thereby, taking $ x_{(\blank)} = g^\dbBlank $ and $ y_{(\dbBlank)}  = \nabla^2 g_\dbBlank $, we are in situation of \cref{lem:commutatorLikeEstimate}---providing the desired estimate.
        \item
            For the second summand \cref{lem:HolderEstimate} gives the desired estimate.
        \item
            In case of the third summand \cref{lem:polyApprox} is applicable due to \cref{eq:mollification_inverseBd,eq:normBd_Holder} providing the estimate
            \[  
                    \left\|
                        b_{\vec{v}} - b_{\vec{v}, \flowVar}
                    \right\|_{\Lebesgue{\infty}} 
                \leq
                    C'
                    \left\| 
                        \left(
                          \nabla g_\dbBlank, \nabla^{\leq 1} g^\dbBlank
                        \right)
                    -
                        \left(
                          \nabla P_\flowVar g_\dbBlank, 
                          \nabla^{\leq 1} (P_\flowVar g_\dbBlank)^{-1} 
                        \right)
                    \right\|_{\Lebesgue{\infty}} 
            \]
            where $C' = C'(B, \|(\nabla g_\dbBlank, \nabla^{\leq 1} g^\dbBlank, \nabla P_\flowVar g_\dbBlank, \nabla^{\leq 1} (P_\flowVar g_\dbBlank)^{-1})\|_{\Lebesgue{\infty}})$.
            The estimate we seek follows now from \cref{eq:prop:mollifiedCurv mollifyBd,eq:mollification_inverseDiffBd}.
    \end{itemize}
    Thus \cref{eq:ansatz_restate} becomes $
            \Riem_{P_\flowVar (g_\dbBlank)}(x)(\vec{v})
	   \leq 
    	        \sup_{y\in \Openball{x}{\flowVar}}  
    	            \Riem_{g_\dbBlank}(y) (\vec{v})
	        + 
	            C\flowVar^\beta
    $
    for a constant $C$ and $\beta \leq \alpha - \nicefrac{n}{p} $ as in the claim of the theorem.
\end{proof}

\begin{proof}[Proof of \cref{thm:mollifiedCurv convexCombi}]
    \begin{subequations}
    For the proof we examine the metric tensor on some chart $\psi$ with $\|\psi\|_{\Sobolev{2}[p], r}^{\textnormal{harm}} \leq Q$, 
    e.g.\ $ \psi = \psi_i $ for one $ i \in I $.
    For any metric tensor $\tilde{g}$ on $M$ representation with respect to $\psi$-coordinates is indicated by $\tilde{g}_\dbBlank$---or $\tilde{g}^\dbBlank$ for the inverse,
    e.g.\ $ 
        g^{[\flowVar]\dbBlank} = \bigl(g^{[\flowVar]}\bigr)^{-1}
    $.
    All sums ${\sum_i}$ range over the entire index set $I$.
    By abuse of notation we write $ \varrho_i = \varrho_i \circ \psi $.
    We define
    \begin{align}
            g^{\Delta \flowVar, \psi_i}
        &\coloneqq
            g - g^{[\flowVar, \psi_i]}
    .
    \\
    \intertext{Further we agree on the shorthands}
            A_{\vec{v}}(x,y)
        &\coloneqq
            A(x,y)(\vec{v})
    \\
            b_{\vec{v}}^{[\flowVar]}
        &\coloneqq
            B(
                \nabla g^{[\flowVar]}{}_\dbBlank, 
                \nabla^{\leq 1} g^{[\flowVar]\dbBlank}
            )(\vec{v})
    \\
            b_{\vec{v}}^{[\flowVar, \psi_i]}
        &\coloneqq
            B(
                \nabla g^{[\flowVar, \psi_i]}{}_\dbBlank, 
                \nabla^{\leq 1} g^{[\flowVar, \psi_i]\dbBlank}
            )(\vec{v})
    \\
    \label{eq:C_rho}
            C_\varrho
        &\coloneqq
            \sum\nolimits_i \|\varrho_i \circ \psi_i \|_{\Hoeld{2}}
    \\
            \bar{v}
        &\coloneqq
            \left\|
            \begin{multlined}(
            v^\psi{}_{11} v^\psi{}_{21} v^\psi{}_{31} v^\psi{}_{41}
            \ldots, 
            v^\psi{}_{1n} v^\psi{}_{21} v^\psi{}_{31} v^\psi{}_{41}
            \\\ldots, 
            v^\psi{}_{1n} v^\psi{}_{2n} v^\psi{}_{3n} v^\psi{}_{4n}
            )
            \end{multlined}
            \right\|_{\Lebesgue{\infty}}
    \text{.}
    \end{align}
    where $v^\psi \coloneqq \psi^* v$.
    \end{subequations}

    \begin{subequations}
    Observe that due to \cref{eq:transition_estimate} $
            \dist*
                { \supp (\varrho_i) }
                {\setBd\left(\psi^{-1} \circ \psi_i(\Openball{0}{r})\right)}
                _{\mathrm{H}}
        \leq
            e^{-2Q} \nicefrac{r}{4}
    $ for all $i\in I$. By \cref{eq:transition_estimate} the estimates \cref{eq:mollification_inverseBd,eq:mollification_inverseDiffBd} imply
    \begin{align}
    \label{eq:thm:mollifiedCurvConv g_bounds1}
            \left\|
                g^{[\flowVar, \psi]\dbBlank }\big|_{\supp \varrho_i}
            \right\|_{\Hoeld{1}[\alpha]}
        &\leq 
            C_{n,Q,r,\alpha}
    \text{,}&
            \left\| 
                g^{[\flowVar]\dbBlank }
            \right\|_{\Hoeld{1}[\alpha]}
        &\leq 
            C_{n,Q,r,\alpha, C_\varrho}
    \text{,}
    \\
    \label{eq:thm:mollifiedCurvConv g_bounds2}
            \left\|
                \left(g^{\Delta \flowVar, \psi_i}\right)^\dbBlank
                \big|_{\supp \varrho_i}
            \right\|_{\Hoeld{1}}
        &\leq 
            C_{n,r,\alpha, Q}' \flowVar^\alpha
    \text{,}&   
            \left\|
                \left(g^{\Delta \flowVar}\right)^\dbBlank
            \right\|_{\Hoeld{1}}
        &\leq 
            C_{n,r,\alpha, Q, C_\varrho}' \flowVar^\alpha
    \end{align}
    for every $i \in I$ and respective constants.
    Moreover, by \cref{eq:transition_estimate} we have the estimate 
    \begin{align}
            \left\| 
                \nabla^2 g^{[\flowVar, \psi_i]\dbBlank }
            \right\|_{\Lebesgue{\infty}}
        &=
            \left\|
                \nabla^2  (\psi_i^{-1} \circ \psi)^* P_\flowVar(\psi_i^* g)
            \right\|_{\Lebesgue{\infty}}
    \nonumber 
    \\
        &\leq
            C_n\| \psi_i^{-1} \circ \psi \|_{\Hoeld{3}}^2
            \|\nabla^2 P_\flowVar(\psi_i^* g)\|_{\Lebesgue{\infty}}
    \nonumber
    \\
        &=
            C_n\| \psi_i^{-1} \circ \psi \|_{\Hoeld{3}}^2
            \|P_\flowVar(\nabla^2 \psi_i^* g)\|_{\Lebesgue{\infty}}
    &&\text{by \cref{eq:convolution_deriv}}
    \nonumber
    \\
        &\leq
            C_{n, \alpha, Q, r}
            \|\varphi_\flowVar\|_{\Lebesgue{q}}
            \|\nabla^2 \psi_i^* g\|_{\Lebesgue{p}}
    &&\text{by \cref{eq:transition_estimate,eq:YoungIneq}}
    \nonumber
    \\
    \label{eq:thm:mollifiedCurvConv 2ndDerBd}
        &\leq
            C_{n, \alpha, Q, r,p} \flowVar^{-n/p}
    &&\text{by \cref{eq:mollifier_estimate}.}
    \end{align}
    \end{subequations}

    Before starting with ansatz \cref{eq:ansatz}, we decompose $\Riem_{g^{[\flowVar]}}$ into a convex combination of functions.
    To this end observe
    \begin{align*}
            \Riem_{g^{[\flowVar]}}(\vec{v})
        &=
                A_{\vec{v}}\left( 
                    g^{[\flowVar]\dbBlank}, \nabla^2 g^{[\flowVar]}{}_\dbBlank 
                \right) 
            + 
                b_{\vec{v}}^{[\flowVar]}
    \\
        &=
                A_{\vec{v}}\left( 
                    g^{[\flowVar]\dbBlank}, 
                    \nabla^2 \sum\nolimits_i  \varrho_i 
                        g^{[\flowVar, \psi_i]}{}_\dbBlank
                \right) 
            + 
                b_{\vec{v}}^{[\flowVar]}
    \\
        &=
                A_{\vec{v}}\left( 
                    g^{[\flowVar]\dbBlank}, 
                    \nabla^2 \sum\nolimits_i  \varrho_i 
                        g_\dbBlank
            -
                \sum\nolimits_i 
                     \nabla^2   \varrho_i 
                        g^{\Delta \flowVar, \psi_i}{}_\dbBlank
                \right) 
            + 
                b_{\vec{v}}^{[\flowVar]}
    \\
        &=
                A_{\vec{v}}\left( 
                    g^{[\flowVar]\dbBlank}, 
                        \nabla^2 g_\dbBlank
                    -
                        \sum\nolimits_i
                      \nabla^2  \varrho_i g^{\Delta \flowVar, \psi_i}{}_\dbBlank
                \right) 
            + 
                b_{\vec{v}}^{[\flowVar]}
    \\
        &=
                A_{\vec{v}}\left( 
                    g^{[\flowVar]\dbBlank}, 
                    \sum\nolimits_i  \varrho_i \nabla^2  g_\dbBlank
                -
                    \sum\nolimits_i
                      \nabla^2  \varrho_i g^{\Delta \flowVar, \psi_i}{}_\dbBlank
                \right) 
            + 
                b_{\vec{v}}^{[\flowVar]}
    \\
        &=
                \sum\nolimits_i  A_{\vec{v}}\left( 
                    g^{[\flowVar]\dbBlank}, 
                    \varrho_i \nabla^2  g_\dbBlank
                -
                    \nabla^2  \varrho_i g^{\Delta \flowVar, \psi_i}{}_\dbBlank
                \right) 
            + 
                b_{\vec{v}}^{[\flowVar]}
    \\
        &=
                \sum\nolimits_i  A_{\vec{v}}\left( 
                    g^{[\flowVar]\dbBlank}, 
                      \varrho_i \nabla^2  
                        \left(
                          g^{[\flowVar, \psi_i]}
                          +
                          g^{\Delta \flowVar, \psi_i}
                        \right)_\dbBlank
                    -
                      \nabla^2  \varrho_i g^{\Delta \flowVar, \psi_i}{}_\dbBlank
                \right) 
            + 
                b_{\vec{v}}^{[\flowVar]}
    \\
        &=
            \sum\nolimits_i
                A_{\vec{v}}\left(
                    g^{[\flowVar]\dbBlank}, 
                    \varrho_i \nabla^2 g^{[\flowVar, \psi_i]}{}_\dbBlank
                -
                    \left(\nabla^2 \varrho_i 
                        g^{\Delta \flowVar, \psi_i}{}_\dbBlank
                    -
                    \varrho_i \nabla^2 g^{\Delta \flowVar, \psi_i}{}_\dbBlank                    \right)
                \right)
            +
                b_{\vec{v}}^{[\flowVar]}
    \text{.}
    \end{align*}
    Thus for a dimension-dependent constant $C_n$ and using definition \cref{eq:maxNorm}
    \begin{align*}
    \MoveEqLeft[2]
            \left\|
                    \Riem_{g^{[\flowVar]}}(\vec{v})
                -
                  \sum\nolimits_i A_{\vec{v}}\left(
                    g^{[\flowVar]\dbBlank}, 
                    \varrho_i \nabla^2 g^{[\flowVar, \psi_i]}{}_\dbBlank
                  \right)
                -
                  b_{\vec{v}}^{[\flowVar]}
            \right\|_{\Lebesgue{\infty}}
    \\
        &\leq
            \sum\nolimits_i 
                \left\|
                    A_{\vec{v}}\left(
                        g^{[\flowVar]\dbBlank}, 
                    \nabla^2 \varrho_i 
                        g^{\Delta \flowVar, \psi_i}{}_\dbBlank
                    -
                    \varrho_i \nabla^2 
                    g^{\Delta \flowVar, \psi_i}{}_\dbBlank
                \right)
                \right\|_{\Lebesgue{\infty}}
    \\
        &\leq
            C_n\bar{v}\sum\nolimits_i 
                \left\|
                    g^{[\flowVar]\dbBlank}
                \right\|_{\Lebesgue{\infty}}
                \left\|
                    \nabla^2 \varrho_i 
                        g^{\Delta \flowVar, \psi_i}{}_\dbBlank
                    -
                    \varrho_i \nabla^2 
                    g^{\Delta \flowVar, \psi_i}{}_\dbBlank
                \right\|_{\Lebesgue{\infty}}
    \\
        &\leq
            C_{n, Q, r, \alpha}' \bar{v} \sum\nolimits_i
                \left\|
                    \nabla^2 \varrho_i 
                        g^{\Delta \flowVar, \psi_i}{}_\dbBlank
                    -
                    \varrho_i \nabla^2 
                    g^{\Delta \flowVar, \psi_i}{}_\dbBlank
                \right\|_{\Lebesgue{\infty}}
    \\
        &\leq
            C_{n, Q, r, \alpha}' \bar{v} \sum\nolimits_i
                \left(
                    \left\| 
                        \nabla^2 \varrho_i \right\|_{\Lebesgue{\infty}}
                    \left\|
                      g^{\Delta \flowVar, \psi_i}{}_\dbBlank
                    \right\|_{\Lebesgue{\infty}}
                  +
                    2\left\|
                      \nabla \varrho_i
                    \right\|_{\Lebesgue{\infty}}
                    \left\| 
                      \nabla g^{\Delta \flowVar, \psi_i}{}_\dbBlank
                    \right\|_{\Lebesgue{\infty}}
                \right)
    \\
        &\leq
            2 C_{n, Q, r, \alpha}' \bar{v}  C_\varrho \sum\nolimits_i 
                \left\|
                    g^{\Delta \flowVar, \psi_i}{}_\dbBlank
                \right\|_{\Sobolev{1}[\infty]}
    \end{align*}
    where we used \cref{eq:thm:mollifiedCurvConv g_bounds1} in the antepenultimate step, and 
    definition \cref{eq:C_rho} in the last step. 
    As $ 
            g^{\Delta \flowVar, \psi_i}{}_\dbBlank 
        =
            \psi^* \linebreak[2] (\psi_i^{-1})^* \linebreak[2] (\psi_i^*g)^{\Delta \flowVar}
        =
            (\psi \circ\nolinebreak[3] \psi_i^{-1})^*\linebreak[2] (\psi_i^*g)^{\Delta \flowVar}
    $,
    combining estimate \cref{eq:transition_estimate} and \cref{eq:prop:mollifiedCurv mollifyBd}
    gives
    \[
            \left\|
                    \Riem_{g^{[\flowVar]}} (\vec{v})
                -
                  \sum\nolimits_i A_{\vec{v}}\left(
                    g^{[\flowVar]\dbBlank}, 
                    \varrho_i \nabla^2 g^{[\flowVar, \psi_i]}{}_\dbBlank
                  \right)
                -
                  b_{\vec{v}}^{[\flowVar]}
            \right\|_{\Lebesgue{\infty}}
        \leq
            C' \bar{v} \flowVar^\alpha
    \]
    for a constant $C' = C'(n,Q,r, \{\varrho_i\circ\psi_i\}_{i\in I}, \alpha)$.
    Finally, this gives
    \begin{align}
    \nonumber
            \Riem_{g^{[\flowVar]}} (x) (\vec{v})
        &\leq
              \sum\nolimits_i 
                A_{\vec{v}}\left(
                  g^{[\flowVar]\dbBlank}, 
                  \varrho_i \nabla^2 g^{[\flowVar, \psi_i]}{}_\dbBlank
                \right)(x)
            +
                b_{\vec{v}}^{[\flowVar]}(x)
            +
                C'\flowVar^\alpha
    \\
    \nonumber
        &\leq
            \sum\nolimits_i \varrho_i \left(
                  A_{\vec{v}}\left(
                    g^{[\flowVar]\dbBlank}, 
                    \nabla^2 g^{[\flowVar, \psi_i]}{}_\dbBlank
                  \right)
                +
                  b_{\vec{v}}^{[\flowVar]}
            \right)(x)
            +
                C'\flowVar^\alpha
    \\
    \label{eq:ansatz_convBombi}
        &\leq
            \sup_{\substack{ 
                        i \in I \\ 
                    \dist{x}{\psi_i(0)}_g < \nicefrac{r}{2} }} 
                \left(
                  A_{\vec{v}}\left(
                    g^{[\flowVar]\dbBlank}, 
                    \nabla^2 g^{[\flowVar, \psi_i]}{}_\dbBlank
                 \right)
              +
                b_{\vec{v}}^{[\flowVar]}
            \right)(x)
            +
                C'\flowVar^\alpha
    .
    \end{align}
    This means that it is sufficient to find an estimate
    \begin{equation*}
                \left(
                      A_{\vec{v}}\left(
                        g^{[\flowVar]\dbBlank}, 
                        \nabla^2 g^{[\flowVar, \psi_i]}{}_\dbBlank
                      \right)
                    +
                      b_{\vec{v}}^{[\flowVar]}
                \right)(x)
            -
                \Riem_g(x)(\vec{v})
        \leq
            C'' \bar{v} \flowVar^\alpha
    \end{equation*}
    for every $i\in I$.
    
    We are now in a position to apply an estimate similar to the one from the proof of \cref{prop:mollifiedCurv}:
    \begin{multline}\label{eq:thm:mollifiedCurv ConvEstimate}
                \left(
                    A_{\vec{v}}\left(
                      g^{[\flowVar]\dbBlank}, 
                      \nabla^2 g^{[\flowVar, \psi_i]}{}_\dbBlank
                \right)
              +
                b_{\vec{v}}^{[\flowVar]}
            \right)(x)
        \leq
                \Riem_{g^{[\flowVar, \psi_i]\dbBlank} }(x)(\vec{v})
            + {}
        \\
            {}+
                \left \|A_{\vec{v}}\left(
                  g^{[\flowVar]\dbBlank} - g^{[\flowVar, \psi_i]\dbBlank},
                  \nabla^2 g^{[\flowVar, \psi_i]}{}_\dbBlank
                \right)\right\|_{\Lebesgue{\infty}}
            +
                \left\|
                    b_{\vec{v}}^{[\flowVar]} 
                  - b_{\vec{v}}^{[\flowVar, \psi_i]}
                \right\|_{\Lebesgue{\infty}}
    \end{multline}
    The proof is reduced to seeking bounds for the second and third summand.
    
    With regard to the second summand we estimate for a constant $C_n$ using definition \cref{eq:maxNorm}
    \begin{align*}
    \MoveEqLeft[2]
            \left \|A_{\vec{v}}\left(
              g^{[\flowVar]\dbBlank} - g^{[\flowVar, \psi_i]\dbBlank}, 
              \nabla^2 g^{[\flowVar, \psi_i]}{}_\dbBlank
            \right)\right\|_{\Lebesgue{\infty}}
    \\
        &\leq  
            C_n
            \bar{v}
            \left\|
                g^{[\flowVar]\dbBlank} - g^{[\flowVar, \psi_i]\dbBlank}
            \right\|_{\Lebesgue{\infty}}
            \left\| 
                \nabla^2 g^{[\flowVar, \psi_i]}{}_\dbBlank
            \right\|_{\Lebesgue{\infty}}
    \\
        &\leq
            C_n  
            \bar{v}
            \left(
                \left \|
                  \left(g^{\Delta \flowVar}\right)\dbBlank
                \right\|_{\Lebesgue{\infty}}
              +
                \left \|
                  \left(g^{\Delta\flowVar, \psi_i}\right)^\dbBlank
                \right\|_{\Lebesgue{\infty}}
            \right)
            \left\| 
                \nabla^2 g^{[\flowVar, \psi_i]}{}_\dbBlank
            \right\|_{\Lebesgue{\infty}}
    \\
        &\leq
            C_n
            \bar{v}
            \left(
                C_{n,r,\alpha, Q}' \flowVar^\alpha
              +
                C_{n,r,\alpha, Q, C_\varrho}' \flowVar^\alpha
            \right)
            C_{n,\alpha,Q,r,p} \flowVar^{n/p}
    \end{align*}
    where we used \cref{eq:thm:mollifiedCurvConv g_bounds1,eq:thm:mollifiedCurvConv 2ndDerBd} in the last step.
    To summarize: the second summand vanishes with modulus $\flowVar^{\alpha - n/p}$.
    
    As for the third summand by \cref{lem:polyApprox} we have a constant $
          C 
        = 
          C\bigl( 
            \bigl\| \bigl( 
                \nabla g^{[\flowVar]}{}_\dbBlank,
                \nabla^{\leq 1} g^{[\flowVar]}{}^\dbBlank
            \bigr)\bigr\|_{\Lebesgue{\infty}} 
          \bigr) 
    $ such that we can estimate
    \begin{align*}
            \left\|
                b_{\vec{v}}^{[\flowVar]} 
              - b_{\vec{v}}^{[\flowVar, \psi_i]}
            \right\|_{\Lebesgue{\infty}}
        &\leq
            C
            \bar{v}
            \Bigl\| 
                \Bigl(
                    \nabla g^{[\flowVar]}{}_\dbBlank,
                    \nabla^{\leq 1} g^{[\flowVar]}{}^\dbBlank
                \Bigr) 
              - \Bigl(
                    \nabla g^{[\flowVar, \psi_i]}{}_\dbBlank,
                    \nabla^{\leq 1} g^{[\flowVar, \psi_i]}{}^\dbBlank
                \Bigr)
            \Bigr\|_{\Lebesgue{\infty}}
    \\
        &\leq
            C 
            \bar{v}
            \Bigl(
                \left\| 
                    g^{[\flowVar]}{}^\dbBlank 
                  - g^{[\flowVar, \psi_i]}{}^\dbBlank
                \right\|_{\Sobolev{1}[\infty]}
              +
                \left\|
                    \nabla g^{[\flowVar]}{}_\dbBlank
                  - \nabla g^{[\flowVar, \psi_i]}{}_\dbBlank
                \right\|_{\Lebesgue{\infty}}
            \Bigr)
    \text{.}
    \end{align*}
    Finally, by \cref{eq:thm:mollifiedCurvConv g_bounds2} we have a bound
    \begin{align*}
            \left\| 
                g^{[\flowVar]}{}^\dbBlank 
              - g^{[\flowVar, \psi_i]}{}^\dbBlank
            \right\|_{\Sobolev{1}[\infty]}
        &\leq
            \left\| 
                g^{[\flowVar]}{}^\dbBlank 
              - g^\dbBlank
            \right\|_{\Sobolev{1}[\infty]}
            +
            \left\| 
                g^\dbBlank 
              - g^{[\flowVar, \psi_i]}{}^\dbBlank
            \right\|_{\Sobolev{1}[\infty]}
    \\
        &\leq
            \left(
                C_{n,r,\alpha, Q}' + C_{n,r,\alpha, Q, C_\varrho}' 
            \right)
            \flowVar^\alpha
    \\
    \intertext{and by \cref{eq:transition_estimate,eq:prop:mollifiedCurv mollifyBd} there is a bound}
            \left\|
                \nabla g^{[\flowVar]}{}_\dbBlank
              - \nabla g^{[\flowVar, \psi_i]}{}_\dbBlank
            \right\|_{\Lebesgue{\infty}}
        &\leq
            \left\|
                \nabla g^{[\flowVar]}{}_\dbBlank
              - \nabla g_\dbBlank
            \right\|_{\Lebesgue{\infty}}
            +
            \left\|
                \nabla g_\dbBlank
              - \nabla g^{[\flowVar, \psi_i]}{}_\dbBlank
            \right\|_{\Lebesgue{\infty}}
    \\
        &\leq
            \left(C_{n,\alpha,Q,C_\varrho} + C_{n,r,\alpha, Q}' \right)
            \flowVar^\alpha
    \text{.}
    \end{align*}
    Thus combining \cref{eq:ansatz_convBombi} with \cref{eq:thm:mollifiedCurv ConvEstimate} we obtain
    \begin{align*}
            \Riem_{g^{[\flowVar]}} (x) (\vec{v})
        &\leq
                \sup_{\substack{ 
                        i \in I \\ 
                        \dist{x}{\psi_i(0)}_g < \nicefrac{r}{2} }}
                    \Riem_{g^{[\flowVar, \psi_i]\dbBlank} }(x)(\vec{v})
            +
                C \bar{v} \flowVar^\alpha
    \\
        &\leq
            \sup
                \setBuilder*
                    {\Riem_g (y) (\vec{v})}
                    {y \in \bigcup\nolimits_{\substack{ 
                        i \in I \\ 
                        \dist{x}{\psi_i(0)}_g < \nicefrac{r}{2} }} 
                        \psi_i(\Openball{0}{\flowVar })
                    }
             +
                C \bar{v} \flowVar^\beta
    \\
        &\leq
                \sup_{ y \in \Openball{x}{e^Q \flowVar } } 
                    \Riem_g (y) (\vec{v})
             +
                C \bar{v} \flowVar^\beta
    \end{align*}
    where in the penultimate step we used \cref{prop:mollifiedCurv} which is applicable since $ \Riem_{g^{[\flowVar, \psi_i]\dbBlank} }(\vec{v}) = \Riem_{g^{[\flowVar, \psi_i] } }(\vec{v}) $ is coordinate independent; and in the last step the distance estimate \cref{eq:lengthEstimate} came to hand.
    By \cref{eq:normBd}, we further have $\bar{v} \leq n^4 \cdot \Bigl( \sqrt{ e^{2Q} } \Bigr)^4 \|\vec{v}\|_g $.
    This implies the bound we sought.
\end{proof}

\section{Consequence for sectional curvature}\label{sec:consequence-for-sec}

This section is devoted to the main result of the paper, which is again phrased in terms of Sobolev chart norms $\|\blank\|_{\Sobolev{2}[p], r}^{\textnormal{harm}}$,  see \cref{eq:normBd_Sobolev}.

\ThmMollifiedSec*

A glance at \cref{eq:absoluteRicciBd} gives immediately:

\corInjRad*

\corRegidity*
 
 \begin{proof}
     Each of \cref{eq:topAssumptionAspher,eq:topAssumptionEvenDim,eq:dim3,eq:quaterPinched} implies a uniform positive lower bound on the injectivity radius \cite[6]{Tuschmann00}.
 \end{proof}

The first step to prove \cref{thm:mollifiedSec} is to introduce the notion of a locally $N$-finite cover: a cover $\{U_i\}_{i \in I}$ of a space $X$ is \definiendum{locally $N$-finite} if every point of $x$ is contained in at most $N$ members of $\{U_i\}_I$.
By help of this terminology one can reduce \cref{thm:mollifiedSec} to the following claim:

\begin{lemma}\label{lem:mollifiedSec}
    Let $p > 2n $, $r, Q> 0$, $\beta\in(0,1-\nicefrac{2n}{p}) $ and $N$ be a natural number.
    Then there is $ \flowVarMax' > 0$ such that for any smooth riemannian manifold $M$ with $
        \|(M, g)\|_{\Sobolev{2}[p], r}^{\textnormal{harm}}  \leq Q 
    $ and any 
    \[
        \{\psi_i\colon (\Openball{0}{r}, 0) \to (M,\psi_i(0))\}_{i\in I}
    ,\quad
        \vec{\varrho} \coloneqq \{\varrho_i\colon M \to [0,1]\}_{i\in I}
    \]
    a cover of $M$ by charts and a corresponding partition of unity with
    \begin{itemize}
        \item 
            $\|\psi_i\|_{\Sobolev{2}[p], r}^{\textnormal{harm}} \leq Q$,
        \item 
            $\setBuilder{\psi_i(\Openball{0}{ r })}{i\in I}$ is locally $N$-finite,
        \item
            $M$ is covered by $\{ \psi_i( \Openball{0}{e^{-Q}r/2} )\}_{i\in I}$, and
        \item 
            $\|\varrho_i \circ \psi_i \|_{\Hoeld{2}} < C_I$ for all $i\in I$ and a constant $C_I$ 
    \end{itemize}
    the mollified metrics $ g^{[\flowVar]} $
    have at any $x\in M$ sectional curvature $ \Sec_{ P_{\flowVar } g }(x) $ in the interval
    \cref{eq:secInterval}
    for all $\flowVar \in (0,\flowVarMax']$,
    where $C = C(n,Q,r,p)$.
\end{lemma}

\begin{proof}[Proof of \cref{thm:mollifiedSec} using \cref{lem:mollifiedSec}]
    By \textcite[\citeProposition 5.2]{Gromov07} every Gromov-Hausdorff compact class $\mathcal{M}$ of (isometry classes of) pointed metric spaces has the following property:
    for every space $ (M,d,\Basepoint) \in \mathcal{M} $ the maximal number of disjoint closed balls of radius $\varepsilon$ that fit into $\Closedball{\Basepoint}{R}$ is bounded by a finite number $ N(\varepsilon, R)$.
    By Sobolev's inequality \cref{eq:SobolevIneq} we have the inclusion $
       \mathcal{M}(\Sobolev{2}[p] \leq_r Q) \subset  \mathcal{M}(\Hoeld{1}[\alpha] \leq_r C Q)
    $ for $ \alpha = \nicefrac{n}{p} $ and a constant $C = C(n,p)$.
    Hence by \cref{eq:fundamentalThm} the space $\mathcal{M}(\Sobolev{2}[p] \leq_r Q)$ is precompact in the Gromov-Hausdorff topology and Gromov's result is applicable.
    
    Now we apply Gromov's result to $\varepsilon = re^{-2Q}/5$ and $R=2re^{Q}$ obtaining a bound $N$.
    Let $\{\Closedball{x_i}{\varepsilon}\}_{i\in I}$ be some maximal disjoint system of $\varepsilon$-balls in $ (M, g,\Basepoint) \in \mathcal{M}(\Sobolev{2}[p] \leq_r Q) $.
    The balls $\{ \Openball{x_i}{re^{-2Q}/2 } \}_{i\in I}$ cover $M$ due to maximality of the system.
    On the other hand for any $x\in M$ the cardinality of the set $
            \setBuilder{i\in I}{x \in \Openball{x_i}{re^{Q}}} 
        \subset 
            \setBuilder
                {i\in I}
                { \Openball{x_i}{re^{Q}} \subset \Closedball{x}{R} }
    $ is bounded by $N$.
    For each point $x_i$ choose a chart $\psi_i\colon (\Openball{0}{r}, 0) \to (M, x_i)$ with $\|\psi_i\|_{\Sobolev{2}[p], r} \leq Q$.
    By \cref{eq:lengthEstimate}, $
            \Openball{x_i}{r e^{-2Q}/2} 
        \subset 
            \psi_i(\Openball{0}{ e^{-Q}r/2} )
    $ and, hence $\{\psi_i(\Openball{0}{e^{-Q}r/2} )\}_{i\in I}$ is covering.
    By the same estimate, $ \psi_i(\Openball{0}{r}) \subset \Closedball{x_i}{r e^{Q}} $ and, thus, $\{\psi_i(\Openball{0}{r})\}_{i\in I}$ is $N$-finite.
    
    \begin{subequations} 
    To find a suitable partition of unity, choose any bump function $b\colon \R^d\to \R$ with $\supp b \subset \Openball{0}{ \nicefrac{3r}{4} }$ and $b|_{\Openball{0}{r e^{-2Q}/2} } \equiv 1$.
    We define for all $i\in I$
    \begin{align}
    \label{eq:thm:mollifiedSec bi}
            b_i (x)
        &\coloneqq
            \!\begin{cases*}
                b \circ\psi_i^{-1}(x)   & if $ x \in \psi_i(\Openball{0}{r}) $
            \\
                0                       & otherwise
            \end{cases*}
    \\
    \label{eq:thm:mollifiedSec rho}
            \varrho_i(x)   
        &\coloneqq
            \frac{1}{\sum_{j\in I} b_j(x) } b_i(x) .
    \end{align}
    \end{subequations}
    Due to \cref{eq:transition_estimate} there is a uniform $\Hoeld{3}[\beta]$-bound on the transition maps  $\psi_i^{-1}\circ\psi_j$ for $i,j\in I$.
    Moreover, $ b $ is $\Hoeld{3}[\beta']$-bounded since its support is compact.
    Hence there is a uniform $\Hoeld{3}[\beta]$-bound on $ b_i \circ \psi_j = b\circ \psi_i^{-1}\circ\psi_j$ for each $i \in I$.
    Since the denominator in \cref{eq:thm:mollifiedSec rho} is at least 1 at each $x \in M$ for at least one $i$ (namely the $i$ for which $\dist{\psi_i(0)}{x} \leq e^{-Q} \nicefrac{r}{2}$),
    $\varrho_i$ is $\Hoeld{2}[\beta]$-bounded uniformly in $i\in I$ as well.
    This puts us in a situation to apply \cref{lem:mollifiedSec}.
    
    Due to \cref{eq:mollification bd} we have a uniform $\Hoeld{2}[\beta]$-bound on $ P_\flowVarMax(\psi_i^* g) $ and thus a uniform $\Hoeld{2}[\beta]$-bound on the pullback metrics $ (\psi_i^{-1}\circ\psi_j)^* P_\flowVarMax(\psi_i^* g) $ for all $i,j \in I$ as well.
    Together with the bound on the $\varrho_i$'s from last paragraph this implies a uniform $\Hoeld{2}[\beta]$-bound on $\psi_i^* g^{[T]} $ for each $i\in I$.
    Thus the bound \cref{eq:secChartNorm} holds.
\end{proof}

\begin{proof}[Proof of \cref{lem:mollifiedSec}]
    We apply \cref{thm:mollifiedCurv convexCombi} obtaining a constant
    \[
        C_1 = C(n, p, r, Q, \beta, C_I, N)
    \text{.}
    \]
    From $\{\psi_i\}_{i\in I}$, $\{\varrho_i\}_{i\in I}$ we get a mollification of $g$: $g^{[\blank]}\colon [0,\flowVarMax] \to \Gamma(\Sym^{0,2} M) $ with $g^{[0]}\coloneqq g$ and the property
    \begin{equation}\label{eq:lem:mollifiedSec Riem}
            \Riem_{ g^{ [\flowVar ] } }(x) (\vec{v})
          - 
            \sup_{y\in \Openball{p}{e^Q \flowVar }}
                \Riem_g(y) (\vec{v})
        \leq 
            C_1 \|\vec{v}\|_g \flowVar^\beta
    \end{equation}
    for any $x \in M$ and any section $
            \vec{v} 
        \in 
            (\Tangent M)^{\times 3} \times \Tangent^* M $ .
    For convenience we abbreviate
    \begin{align*}
            \langle\blank,\blank\rangle_\flowVar
        &\coloneqq
            \langle\blank,\blank\rangle_{g^{[\flowVar]} }
    ,
    &
            \|\blank\|_\flowVar
        &\coloneqq
            \|\blank\|_{g^{[\flowVar]} }
    ,
    \\
            \Riem_\flowVar
        &\coloneqq
            \Riem_{g^{[\flowVar]}}
    ,
    &
            \Sec_\flowVar
        &\coloneqq
            \Sec_{g^{[\flowVar]}}
    .
    \end{align*}
    We further agree on the following shorthands for intervals: 
    $ [a\pm b]\coloneqq [a-b, a+b] $, $[\pm b] \coloneqq [0\pm b]$, and
    \[
            [f(x)]_{x\in X}
        \coloneqq
            \left[
                \underset{x\in X}{\vphantom{\operatorname{p}}\inf} f(x),
                \sup_{x\in X} f(x)
            \right]
    .
    \]
    Moreover, from multi-linearity of the curvature tensor (use e.g.\ $ \Riem(v,w) = - \Riem(-v,w) $) we get immediately the reversed version of \cref{eq:lem:mollifiedSec Riem}
    \begin{equation*}
            -\Riem_{ g^{ [\flowVar ] } }(x) (\vec{v})
          + 
            \inf_{y\in \Openball{x}{e^Q \flowVar }} \Riem_g(y) (\vec{v})
        \geq 
            - C_1 \|\vec{v}\|_g \flowVar^\beta
    .
    \end{equation*}
    
    We seek a $\flowVarMax' \in (0,\flowVarMax] $ such that the claim of the theorem holds, i.e.
    \[
            \Sec_\flowVar(x) (v, w)
        =
            \frac 
                {\Riem_\flowVar(\vec{v})}
                {
                    \|v\|_\flowVar \|w\|_\flowVar 
                    - \langle v, w \rangle_\flowVar
                }
    ,
    \]
    where $
            \Riem_\flowVar(\vec{v}) 
        = 
            \langle \Riem_\flowVar(x)(v, w)w, v \rangle_\flowVar 
    $, is in the interval
    \begin{equation}\label{eq:lem:mollifiedSec claim}
        \left[
            \Sec_0(y)(v,w)
        \right]_{ \substack{y \in \Openball{x}{e^Q}\\ v,w \in \Tangent_y M }  }
        + \left[\pm C\flowVar^\beta\right]
    \end{equation}
    for all $\flowVar\in(0,\flowVarMax']$, $x \in M$, and $v, w \in \Tangent M $ linear independent.
    Since the sectional curvature depends only on the plane spanned by $v$ and $w$, 
    we can assume without loss of generality that
    \begin{equation*}
        \|v\|_0 = \|w\|_0 = 1
    ,\quad
        \langle v, w\rangle_0 = 0
    .
    \end{equation*}

    Fix some chart $\psi\colon (\Openball{0}{r}, 0) \to (M,\Basepoint) $ with $\|\psi\|_{\Sobolev{2}[p], r}^{\textnormal{harm}}  \leq Q $.
    We extend $v$ and $w$ to sections $v_\psi, w_\psi \in \Gamma \Tangent^{3,1}\psi(\Openball{0}{r})$ using the euclidean identification, i.e. by pushing forward the constant sections $ v $ and $w$ along $\psi$.
    Again $
        \vec{v}_\psi \coloneqq (v_\psi, w_\psi, w_\psi, \langle \blank, v_\psi \rangle_\flowVar)
    $.
    
    By \cref{eq:normBd} we have $
            \|v_\psi\|_{\textnormal{eucl.}}, \|w_\psi\|_{\textnormal{eucl.}}
        \leq
            e^{2Q}
    $.
    Hence $ \|v_\psi\|_g, \|w_\psi\|_g \leq e^{4Q} $.
    Thus $ \|\vec{v}_\psi \|_g \leq e^{4\cdot 4Q} = e^{16Q} $.
    By Sobolev's inequality the entries of the metric tensor are at least Lipschitz with some bound $C_{n,p,r,Q}$.
    Hence we have
    \begin{multline*} 
            (\|v_\psi\|_0 \|w_\psi\|_0 - \langle v_\psi, w_\psi \rangle_0)(\psi(y)) 
        =
            1 - \langle v_\psi, w_\psi \rangle_0(\psi(y))
        \\\in
            \left[ 
                1 \pm C_{n,p,r,Q}e^{2Q} \cdot \dist{0}{y}
            \right]    
    \end{multline*} for $ y \in \Openball{0}{r} $.
    The Lipschitz bound on the metric tensor implies further via \cref{lem:HolderEstimate}
    \[
            (\|v_\psi\|_\flowVar \|w_\psi\|_\flowVar - \langle v_\psi, w_\psi \rangle_\flowVar)
            (\psi(y)) 
        \in
            \left[ 
                1 \pm C_2' \flowVar
            \right]
    \] 
    for $ y \in \Openball{0}{e^Q \flowVar} $ and $\flowVar \in [0,\flowVarMax]$.
    Choose $\flowVarMax'$ so small that $ C_2' \flowVarMax' \leq \nicefrac{1}{2} $.
    Hence we can choose a constant $C_2>0$ such that
    \[
            \frac
                {1}
                {\|v_\psi\|_\flowVar \|w_\psi\|_\flowVar 
                    - \langle v_\psi, w_\psi \rangle_\flowVar}
            (\psi(y))
        \in
            \left[
                1 \pm C_2 \flowVar
            \right]
    \]
    and choose $T' > 0$ so small that $1  - C_2 \flowVar \geq \nicefrac{1}{2}$
    for all $\flowVar \in (0, \flowVarMax']$.
    
    Gathering all estimates above we conclude the proof:
    \begin{align*}
    \MoveEqLeft[2]
            \Sec_\flowVar(x) (v, w)
        =
            \frac 
                {\Riem_\flowVar(\vec{v}_\psi)}
                {
                    \|v_\psi\|_\flowVar \|w_\psi\|_\flowVar 
                    - \langle v_\psi, w_\psi \rangle_\flowVar
                }
            (x)
    \\
        &\in
            \Riem_\flowVar(x)(\vec{v}_\psi)
            \left[ 1 \pm C_2 \flowVar \right]
    \\
        &\subset
            \left(
                [\Riem_0(y) (\vec{v}_\psi)]_{y\in \Openball{x}{e^Q \flowVar }}
                +
                 \left[ \pm C_1 \|\vec{v}_\psi\|_g \flowVar^\beta\right]
            \right)
            \left[ 1 \pm C_2 \flowVar \right]
    \intertext{for some new constant $C'$ this is included in}
        &\subset
            \left(
                [\Riem_0(y) (\vec{v}_\psi)]_{y\in \Openball{x}{e^Q \flowVar }}
            \right)
            \left[ 1 \pm C_2 \flowVar \right]
          +
            \left[\pm C'\flowVar^\beta\right]
    \\
        &\subset
            \left(
                \left[
                    \Sec_0 (y) (v_\psi, w_\psi)
                \right]_{y\in \Openball{x}{e^Q \flowVar }}
            \right)
            \left[ 1 \pm C_2' \flowVar \right]
            \left[ 1 \pm C_2 \flowVar \right]
          +
            \left[\pm C'\flowVar^\beta\right]
    .
    \end{align*}
    This is contained in the interval \cref{eq:lem:mollifiedSec claim} for some constant $C$. This proves the theorem.
\end{proof}

\noarxiv{
    \printbibliography
}
\arxiv{
    \bibliographystyle{amsalpha}
    \bibliography{bibTeX}
}

\end{document}